\documentclass[12pt,reqno]{amsart}
\usepackage{amsthm,amsfonts,amssymb,euscript,fancyhdr}

\usepackage{graphicx} 
\usepackage[font=small,labelfont=bf]{caption} 

\newcommand{\bl}{\begin{lemma}}
\newcommand{\el}{\end{lemma}}
\def\beaa{\begin{eqnarray*}}
\def\eeaa{\end{eqnarray*}}
\def\ba{\begin{array}}
\def\ea{\end{array}}
\def\be#1{\begin{equation} \label{#1}}
\def \eeq{\end{equation}}

\def\a{{\alpha}}

\def\b{{\beta}}
\def\be{{\beta}}
\def\ga{\gamma}

\def\de{\delta}

\def\la{\lambda}
\def\La{\Lambda}

\def\Om{\Omega}
\def\th{\theta}
\def\Th{\Theta}

\def\nab{\nabla}
\def\varep{\varepsilon}

\def\pr{{\partial}}

\def\CC{{\mathcal C}}

\def\HH{{\mathcal H}}
\def\LL{{\mathcal L}}

\def\WW{{\mathcal W}}

\def\SS{{\mathcal S}}

\def\UU{{\mathcal U}}

\def\HHH{{\mathbb{H}}}
\def\UU{{\mathcal{U}}}

\def\RRR{{\mathbb R}}

\def\f12{{\frac 1 2}}

\DeclareMathOperator{\Div}{\mathrm{div}}

\def\half{\frac{1}{2}}

\newcommand{\Qdd}{{{Q \mkern-12mu / \mkern+5mu }}}

\newcommand{\RRRic}{\mathrm{Ric}}

\newcommand{\Rscal}{\mathrm{R}_{scal}}

\newcommand{\ol}{\overline} 
\newcommand{\Divd}{\Div \mkern-17mu /\ }

\newcommand{\Nd}{\nabla \mkern-13mu /\ }

\newcommand{\Ld}{\triangle \mkern-12mu /\ }

\newcommand{\gd}{{g \mkern-8mu /\ \mkern-5mu }}

\newcommand{\prd}{{\pr \mkern-10mu /\ \mkern-5mu }}

\def\ni{\noindent}

\def\tr{\mathrm{tr}}

\def\th{\theta}

\def\f{\widetilde{f}}

\newcommand{\Lied}{\mathcal{L} \mkern-9mu/\ \mkern-7mu}

\newcommand{\ubp}{\underline{B}^+}
\newcommand{\usp}{\underline{S}^+}

\newtheorem{theorem}{Theorem}[section]
\newtheorem{lemma}[theorem]{Lemma}
\newtheorem{proposition}[theorem]{Proposition}
\newtheorem{corollary}[theorem]{Corollary}
\newtheorem{definition}[theorem]{Definition}
\newtheorem{remark}[theorem]{Remark}

\newtheorem{claim}[theorem]{Claim}

\numberwithin{equation}{section}
\begin{document}

\title[Boundary harmonic coordinates]{Boundary harmonic coordinates on manifolds\\  with boundary in low regularity}

\author{Stefan Czimek}

\address{Department of Computer and Mathematical Sciences, University of Toronto at Scarborough, Toronto, Canada}

\begin{abstract} In this paper, we prove the existence of $H^2$-regular coordinates on Riemannian $3$-manifolds with boundary, assuming only $L^2$-bounds on the Ricci curvature, $L^4$-bounds on the second fundamental form of the boundary, and a positive lower bound on the volume radius.

The proof follows by extending the theory of Cheeger-Gromov convergence to include manifolds with boundary in the above low regularity setting. The main tools are boundary harmonic coordinates together with elliptic estimates and a geometric trace estimate, and a rigidity argument using manifold doubling. Assuming higher regularity of the Ricci curvature, we also prove corresponding higher regularity estimates for the coordinates.
\end{abstract}

\maketitle
\tableofcontents


\section{Introduction} \label{sec:introduction}

\ni This paper is concerned with the following question.\\

\emph{Consider a compact smooth Riemannian $3$-manifold $(M,g)$ with boundary. Under which assumptions does there exist around each point in $M$ a local coordinate system $(x^1,x^2,x^3)$ of uniform size in which the metric components $g_{ij}$ are uniformly bounded in $H^2$?} \\

The following is a first version of our main theorem, answering the above question. The precise version is stated in Section \ref{ExactStatement}.
\begin{theorem}[Existence of regular coordinates, version 1] \label{thm:higherintroexharm111}
Let $(M,g)$ be a compact complete\footnote{A smooth Riemannian manifold with boundary is called complete if it is complete as a metric space. } Riemannian $3$-manifold with boundary such that 
\begin{align*}
\RRRic \in L^2(M), \, \Th \in L^4(\pr M), \, r_{vol}(M,1) > 0,
\end{align*}
where $\Th$ denotes the second fundamental form of $\pr M \subset M$ and $r_{vol}(M,1)$ is the volume radius at scale $1$ of $(M,g)$\footnote{\noindent Let $(M,g)$ be a Riemannian $3$-manifold. The volume radius at scale $r$ is defined as
\begin{align*}
r_{vol}(M,r) := \inf\limits_{p \in M}  \inf\limits_{r'< r} \frac{\mathrm{vol}_g \left( B_g(p,r') \right)}{\frac{4\pi}{3} (r')^{3}},
\end{align*}
where $B_g(p,r')$ denotes the geodesic ball of radius $r'$ centered at $p$.} 
Then the following holds.
\begin{enumerate} 
\item {\bf $L^2$-regularity.} There is $\varep_0>0$ such that for every $0< \varep< \varep_0$, there is a radius
\begin{align*}
r=r( \Vert \RRRic \Vert_{ L^2(M) },\Vert \Th \Vert_{ L^4(\pr M)},r_{vol}(M,1), \varep) >0
\end{align*}
such that around every point $p \in M$ on a geodesic ball of radius $r$, there are coordinates $(x^1,x^2,x^3)$ in which
\begin{align*}
(1-\varep) e_{ij} \leq g_{ij} &\leq (1+\varep) e_{ij},
\end{align*}
where $e_{ij}$ denotes the Euclidean metric, and further
\begin{align*}
 \max\limits_{i,j=1,2,3} \Big( r^{-1/2} \Vert \pr g_{ij} \Vert_{L^2} + r^{1/2}\Vert \pr^2 g_{ij} \Vert_{L^2} \Big)<\varep.
\end{align*}
\item {\bf Higher regularity.} Let $m\geq1$ be an integer. Assuming higher regularity of $\RRRic$, we have further the higher regularity estimate
\begin{align*}
\sum\limits_{i,j=1,2,3} \Vert g_{ij} \Vert_{ H^{m+2}} \leq C_{r} \sum\limits_{i=0}^m \Vert \nab^{(i)} \RRRic \Vert_{L^2(M)}  + C_{r,m} \varep.
\end{align*}
Here, as in the rest of the paper, the notation is that a constant $C_{q_1, \dots, q_k}>0$ depends on the quantities $q_1, \dots, q_k$.
\end{enumerate}
\end{theorem}

The proof of Theorem \ref{thm:higherintroexharm111} is based on the so-called Cheeger-Gromov theory of manifold convergence. In the literature, the Cheeger-Gromov convergence theory and existence results for regular coordinates are readily available for 
\begin{itemize}
\item manifolds without boundary in low regularity, see for example \cite{Petersen},
\item manifolds with boundary in higher regularity, see for example \cite{Kodani}, \cite{Anderson2}, \cite{Anderson} and \cite{Perales}. These results assume pointwise bounds on $\RRRic$, while in Theorem \ref{thm:higherintroexharm111} we only assume $L^2$-bounds on $\RRRic$.
\end{itemize}

Our main motivation to prove Theorem \ref{thm:higherintroexharm111} comes from applications to the Cauchy problem of general relativity, see for example the \emph{localised bounded $L^2$-curvature theorem} in \cite{Cz22}. 

The rest of this introduction is organised as follows. In Section \ref{secCheegerGromovhistory}, we introduce the Cheeger-Gromov theory of manifold convergence and give a short historic overview of its development. In Section \ref{discussionbdryharm}, we discuss so-called \emph{boundary harmonic coordinates} which are essential for our study of the Cheeger-Gromov theory for manifolds with boundary in low regularity.


\subsection{The Cheeger-Gromov theory of manifold convergence} \label{secCheegerGromovhistory} 

The Cheeger-Gromov theory of manifold convergence aims to answer the following questions. \\

\emph{Let $ ( (M_i,g_i))_{i\geq1}$ be a sequence of Riemannian manifolds. What does it mean for this sequence to converge to a metric space? Given uniform quantitative bounds on each $(M_i,g_i)$, can one deduce the convergence of a subsequence? What are the weakest necessary uniform bounds to derive a (pre-)compactness result?} \\

Convergence results for sequences $((M_i,g_i))_{i\geq1}$ of Riemannian manifolds are first proved in \cite{CheegerFirst} and \cite{GromovFirst}, see also \cite{CGref2} and \cite{CGref1}, under the assumption of pointwise bounds on the sectional curvature, and are usually refered to as the \emph{Cheeger-Gromov theory of manifold convergence}.

An important development in the context of this paper are \cite{CGrefAnderson2} and \cite{CGref14} which obtain pre-compactness results for sequences of Riemannian $n$-manifolds under the assumption of pointwise bounds on the Ricci tensor, $L^{n/2}$-bounds on the Riemann curvature tensor and that $r_{vol}(M_i,1)>0$. Subsequently, pre-compactness results under the assumption of only uniform $L^{p}$-bounds on the Riemann curvature tensor for $p>n/2$ and a uniform positive lower bound on the volume radius are derived in \cite{Anderson1993} and \cite{Petersen}, see also \cite{Yang}. 

While the above results are for manifolds without boundary, convergence results for sequences of Riemannian manifolds with boundary are studied for example in \cite{Kodani}, \cite{Anderson2}, \cite{Anderson}, \cite{Perales333} and \cite{Perales}. Note that these results assume pointwise estimates on the Ricci tensor. We refer the reader to the survey articles \cite{Sormani} and \cite{Perales222}. 

Compared to the above literature results, in this paper we prove a pre-compactness result for sequences $((M_i,g_i))_{i\geq1}$ of Riemannian $3$-manifolds with boundary under the assumption of uniform $L^2$-bounds on the Ricci tensor and $L^4$-bounds on the second fundamental form of the boundary, and a uniform positive lower bound on the volume radius, see Theorem \ref{thm:main1} with Corollary \ref{corollary:fundamentalharm}. 

The novely in our result is the application of a geometric trace estimate to control the Gauss curvature of the boundary of the manifold in low regularity (see also the next Section \ref{discussionbdryharm}) and a new rigidity argument based on the manifold double. Our presentation follows \cite{Petersen} and \cite{PetersenBook}.
\subsection{Boundary harmonic coordinates} \label{discussionbdryharm}

Our pre-compactness result for the Cheeger-Gromov theory on manifolds with boundary in low regularity is based on so-called \emph{boundary harmonic coordinates}. These were already studied in \cite{Anderson} in a higher regularity context. They are defined as follows. 

On a Riemannian $3$-manifold $(M,g)$, let $(x^1,x^2,x^3)$ be a local coordinate system near the boundary such that $\{ x^{3}=0 \} \subset \pr M$. These coordinates are called \emph{boundary harmonic} if for $i=1,2,3$ and $A=1,2$
\begin{align*}
\triangle_g x^i &=0  \,\,\, \text{in } \{ x^3>0\} \subset M,  \\
\Ld_{\gd} x^{A}&= 0 \,\,\, \text{on }  \{ x^3 =0 \} \subset \pr M.
\end{align*}
Here $\gd$ and $\Ld_{\gd}$ denote the induced metric and Laplace-Beltrami operator on $\pr M$, respectively. 

Boundary harmonic coordinates 
are particularly useful for the analysis of regularity of coordinates because the corresponding metric components $g_{ij}$ and $\gd_{AB}$ satisfy the elliptic equations
\begin{align}
\triangle_g g_{ij} + Q_{ij}(g,\pr g) &= -\RRRic_{ij} \,\,\, \text{in } \{ x^3>0\} \subset M,  \label{eqfito11} \\
\half \Ld_\gd \gd_{AB} + \Qdd_{AB}(\gd,\prd \gd) &= - K \gd_{AB} \,\,\,\text{on }  \{ x^3 =0 \} \subset \pr M, \label{eqfito22}
\end{align}
for $i,j=1,2,3$ and $A,B=1,2$. Here the nonlinearities are schematically given by 
\begin{align*}
Q_{ij} \approx g (\pr g)^2, \Qdd_{AB} \approx \gd (\prd \gd)^2
\end{align*}
where $\prd \in \{ \pr_1, \pr_2\}$, and $K$ denotes the Gauss curvature of the boundary $\pr M \subset M$, see Lemma \ref{lemmaharmexpr}. 

In Section \ref{sec:ExistenceL2}, we study a sequence $(g_i)_{i \geq1}$ of Riemannian metrics in boundary harmonic coordinates which weakly converges to a limit metric $g$ in $H^2$. We prove its strong convergence $g_i \to g$ in $H^2$ by applying the following general sequence of elliptic estimates for a Riemannian metric in boundary harmonic coordinates, see Section \ref{secstrong111} for details.
\begin{enumerate}

\item By a geometric trace estimate (see Section \ref{SectionTraceEstimateForK}), it holds that the Gauss curvature $K$ is bounded in $H^{-1/2}$ on $\pr M$ by an $L^2$-bound on $\RRRic$ and $L^4$-bound on $\Th$. Subsequently, by standard elliptic estimates applied to \eqref{eqfito22} on $\pr M$, the metric components $\gd_{AB}$ on $\{x^3=0\}\subset \pr M$ are controlled in $H^{3/2}$. Then, using that $g_{AB} \vert_{\{ x^3=0\}} =\gd_{AB}$ on $\pr M$, it follows by standard global elliptic estimates applied to \eqref{eqfito11} in $M$ that the metric components $g_{AB}$ are bounded in $H^2$.
\item In boundary harmonic coordinates, the inverse metric components $g^{33}$ and $g^{3A}$ also satisfy elliptic equations analogous to \eqref{eqfito11}, see Lemma \ref{lemmaharmexpr}. Moreover, on the boundary $\{ x^3=0\} \subset \pr M$, Neumann data for the components $g^{33}$ and $g^{3A}$ is explicitly determined by
\begin{align} \begin{aligned}
N \left( g^{33} \right) &= 2 \tr \Th g^{33},  \\
N \left( g^{3A} \right) &=  \tr \Th g^{3A} - \half \frac{1}{\sqrt{g^{33}}} g^{Ai} \pr_i g^{33}, 
\end{aligned} \label{nmdatafieto1} \end{align}
where $N$ denotes the outward-pointing unit normal to $\pr M$, see Lemma \ref{NeumannHarmonic}. Consequently, by standard elliptic estimates for the Neumann problem with Neumann data \eqref{nmdatafieto1}, it follows that the inverse metric components $g^{33}$ and $g^{3A}$ are bounded in $H^2$ in $M$ by an $L^2$-bound on $\RRRic$ and an $L^4$-bound on $\Th$.
\item We note that the above $H^2$-bounds on the metric components $g_{AB}$, $g^{33}$ and $g^{3A}$ imply $H^2$-bounds for all metric components $g_{ij}$, $i,j=1,2,3$. This follows from Kramer's rule and standard product estimates for $H^2$-functions.
\end{enumerate}


\subsection{Overview of the paper}
In Section \ref{sec:definitions}, we introduce notations, function spaces and the Riemannian geometry setting. In Section \ref{ExactStatement}, we state a precise version of our main result. In Section \ref{sec:setup}, we set up the theory of manifold convergence for manifolds with boundary in low regularity. In Sections \ref{sec:ExistenceL2} and \ref{sec:harmregularitymflds}, we prove parts (1) and (2) of Theorem \ref{thm:higherintroexharm}, respectively. In the appendix, we collect the standard elliptic regularity theory used in Sections \ref{sec:setup}, \ref{sec:ExistenceL2} and \ref{sec:harmregularitymflds}.


\subsection{Acknowledgements} This work forms part of my Ph.D. thesis. I am grateful to my Ph.D. advisor J\'er\'emie Szeftel for his kind supervision and careful guidance. Furthermore, I would like to thank Lars Andersson for suggesting the idea to use manifold doubling. This work is financially supported by the RDM-IdF.


\section{Notations, definitions and prerequisites} \label{sec:definitions}
In this section, we introduce notations, definitions and preliminary results that are used in this paper. 


\subsection{Basic notation} \label{sec:basicnotation1}
In this work, uppercase Latin indices run trough $A,B,C,D,E,F=1,2$ and lowercase Latin indices through $a,b,c,d,i,j=1,2,3$. Greek indices run through $\a,\b,\ga,\de,\mu, \nu = 0,1,2, 3$. We tacitly use the Einstein summation convention. We write $ A \lesssim B$ if there exists a universal constant $C>0$ such that $A \leq C B$. 

We denote strong and weak convergence by $\to$ and $\rightharpoonup$, respectively. We abuse notation by keeping the same index while going to a subsequence of a sequence.

Let the closed upper half-space of $\RRR^3$ be denoted by
$$\HHH^+ := \Big\{ x \in \RRR^3 \Big\vert x^3 \geq 0 \Big\},$$ 
and for points $x\in \HHH^+$ and real numbers $r>0$ let
\begin{align} \begin{aligned}
B(x,r) &:= \Big\{ y \in \RRR^3 \Big\vert \vert x-y \vert < r \Big\},&&& S(x,r) &:= \Big\{ y \in \RRR^3 \Big\vert \vert x-y \vert = r \Big\},\\
B^+(x,r) &:= B(x,r) \cap \HHH^+, &&& S^+(x,r) &:= S(x,r) \cap \HHH^+,\\
\ubp(x,r) &:= B^+(x,r) \cap \{ x^3 =0\}, &&& \usp(x,r) &:= S^+(x,r) \cap \{ x^3 =0 \}.\\
\end{aligned} \label{defballscoord}\end{align}  

In given coordinates $(x^1,x^2,x^3)$, let $\prd \in \{ \pr_1, \pr_2 \}$ and $\pr \in \{ \pr_1, \pr_2, \pr_3\}$. 

An open subset of $\RRR^n, n\geq1$, or of $\HHH^+$ has \emph{smooth boundary} if its closure has smooth boundary. A \emph{smooth domain} is an open subset $\Om$ of $\RRR^n, n\geq1$, or $\HHH^+$ which is connected and has smooth boundary $\pr \Om$. Denote by
\begin{align*}
\underline{\Om} := \Om \cap \{ x^3 =0\}.
\end{align*}

We note that $B^+(x,r)$ and $\ubp(x,r)$ have Lipschitz-regular and smooth boundary, respectively.

For $x\in \HHH^+$ and two reals $0<r'<r$, let $\Om_{x,r',r} \subset \HHH^+$ be a smooth domain such that
\begin{align} \label{defomegq3424145}
B^+(x,r') &\subset \subset \Om_{x,r',r} \subset \subset B^+(x,r), \\
\ubp(x,r') & \subset \subset \underline{\Om}_{x,r',r} \subset \subset \ubp(x,r),
\end{align}
and define the corresponding smooth cut-off function $\chi_{x,r',r}: \HHH^+ \to [0,1]$ such that
\begin{align}\begin{aligned} 
\chi_{x,r',r} \vert_{B^+(x,r')} &\equiv 1, \\
 \mathrm{supp } \chi_{x,r',r} &\subset \subset \Om_{x,r',r}. \end{aligned}\label{defchi}
\end{align}
See also the shaded region in Figure 1.
\begin{center}
\includegraphics[height=4cm]{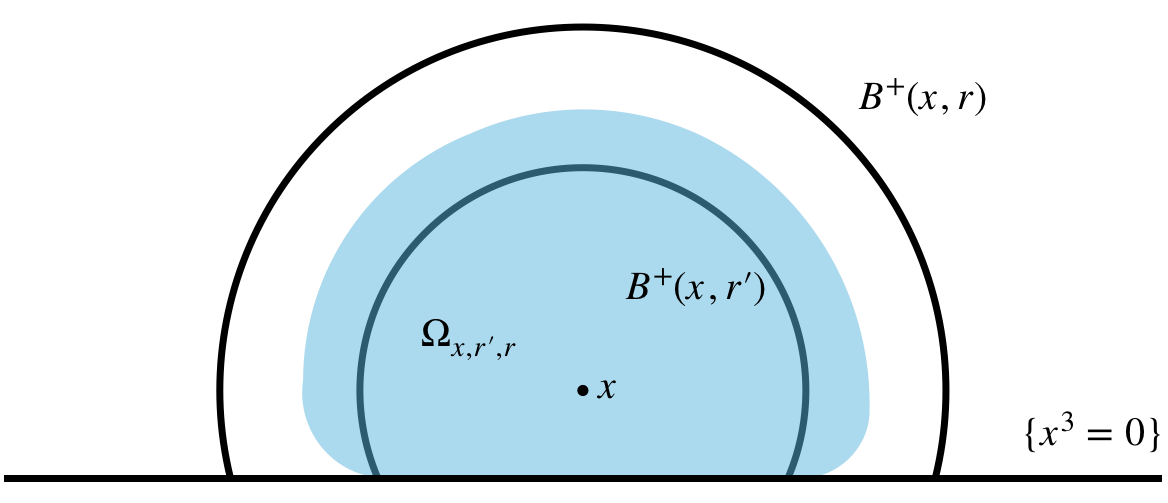} 
\captionof{figure}{The smooth domain $\Om_{x,r',r}$ is depicted as the shaded region.}
\end{center}

\subsection{Function and tensor spaces} \label{sec:FunctionSpaces}

In this section, we introduce the function and tensor spaces that are used in this paper. Let $n\geq 1$ be an integer.
\subsubsection{Continuous and H\"{o}lder-continuous functions}

\begin{definition} \label{defhoeldercont}
Let $m\geq0$ be an integer and $\a \in (0,1]$ a real. Let $\Om \subset \RRR^n$ be an open bounded set and $f$ a scalar function on $\Om$. Define the H\"{o}lder semi-norm by
\begin{align*}
[ f ]_{C^{0,\a}(\ol{\Om})} := \sup\limits_{x\neq y \in \ol{\Om}} \frac{\vert f(x)-f(y) \vert}{\vert x-y \vert^{\a}},
\end{align*}
and let the norm
\begin{align*}
\Vert f \Vert_{C^{m,\a}(\ol{\Om})} := \Vert f \Vert_{C^m(\ol{\Om})} + \max\limits_{\vert \be \vert = m} \left[ \pr^\be f \right]_{C^{0,\a}(\ol{\Om})},
\end{align*}
where 
\begin{align*}
\Vert f \Vert_{C^m(\ol{\Om})} := \max\limits_{\vert \be \vert \leq m}  \, \sup\limits_{x\in \ol{\Om}} \left\vert \pr^\be f \right\vert.
\end{align*}
Here $\be \in \mathbb{N}^n$ is a multi-index and
\begin{align*}
\vert \be \vert := \vert \be_1 \vert + \dots + \vert \be_n \vert, \,\,\,  \pr^\be := \pr_1^{\be_1} \dots \pr_n^{\be_n}.
\end{align*}
Let $C^m(\ol{\Om})$ be the space of $m$-times continuously differentiable functions on $\ol{\Om}$ equipped with the norm $\Vert \cdot \Vert_{C^m(\ol{\Om})}$. Let the H\"older spaces $C^{m,\a}(\ol{\Om})$ be defined by
\begin{align*}
C^{m,\a}(\ol{\Om}) := \left\{ f \in C^m(\ol{\Om}): \Vert f\Vert_{C^{m,\a}(\ol{\Om})} < \infty \right\}.
\end{align*}
\end{definition}


\subsubsection{Fractional Sobolev spaces} In this section we define the fractional Sobolev spaces $W^{s,p}(\RRR^n)$ for reals $s\in \RRR$ and $1<p<\infty$, and summarise basic properties. For more details, see for example \cite{Stein} and \cite{Adams}.

\begin{definition}[Fractional Sobolev spaces] Let $\Om \subset \RRR^n$ be an open set, and let $s\in \RRR$ and $1<p<\infty$ be two reals. The function space $W^{s,p}(\Om)$ is defined as
\begin{itemize}
\item If $s \geq0$ integer,
\begin{align*}
W^{s,p}(\Om) := \left\{ f \in L^p(\Om): \Vert f \Vert^p_{W^{s,p}(\Om)} < \infty \right\},
\end{align*}
where
\begin{align*}
\Vert f \Vert^p_{W^{s,p}(\Om)} := \sum\limits_{\vert \be \vert\leq s} \Vert \pr^\be f \Vert^p_{L^p(\Om)}.
\end{align*}
\item If $s \in (0,1)$,
\begin{align*}
W^{s,p}(\Om) := \left\{ f \in L^p(\Om): \Vert f \Vert^p_{W^{s,p}(\Om)} < \infty \right\},
\end{align*}
where
\begin{align*}
\Vert f \Vert^p_{W^{s,p}(\Om)} := \Vert f \Vert_{L^p(\Om)}^p + \int\limits_{\Om}\int\limits_{\Om} \frac{\vert f(x) -f(y) \vert^p}{\vert x-y \vert^{n+ sp}} dx dy.
\end{align*}
\item If $s = k + \th$ with $k\geq0$ integer, $\th \in (0,1)$,
\begin{align*}
W^{s,p}(\Om) := \left\{ f \in L^{p}(\Om) : \Vert f \Vert_{W^{s,p}(\Om)} < \infty \right\}.
\end{align*}
where
\begin{align*}
 \Vert f \Vert_{W^{s,p}(\Om)} := \Vert f \Vert_{W^{k,p}(\Om)} +\sum\limits_{\vert \a \vert=k} \Vert \pr^\a f \Vert_{W^{\th,p}(\Om)}.
\end{align*}
\end{itemize}
For $s<0$, let $W^{s,p}(\Om)$ be the dual space to the closure of compactly supported smooth functions with respect to the topology of $W^{-s, \frac{p}{p-1}}(\Om)$, denoted by $W_0^{-s,\frac{p}{p-1}}(\Om)$.
\end{definition}
For $s>0$ not integer, $W^{s,p}$ is called \emph{Sobolev-Slobodeckij} space in the literature.

\begin{remark} For $s\in \RRR$ and $p=2$, we have the identification
\begin{align*}
W^{s,p}(\RRR^n)=H^s(\RRR^n),
\end{align*}
where $H^s(\RRR^n)$ is the Sobolev space defined 
\begin{itemize}
\item for $s<0$ by
\begin{align*}
H^s(\RRR^n) := \mathrm{closure}\left( \left\{ f \in \mathcal{S'}(\RRR^n) : \widehat{f} \in L^1_{loc}(\RRR^n), \Vert (1+ \vert \xi \vert^2)^{s/2} \widehat{f} \, \Vert_{L^2(\RRR^n)} < \infty \right\} \right).
\end{align*}
\item for $s\geq0$ by
\begin{align*}
H^s(\RRR^n) := \mathrm{closure}\left( \left\{ f \in \mathcal{S}(\RRR^n) : \Vert (1+ \vert \xi \vert^2)^{s/2} \widehat{f} \, \Vert_{L^2(\RRR^n)} < \infty \right\} \right).
\end{align*}
\end{itemize}
In the above, the closure is taken with respect to the indicated norm, and the space $\mathcal{S}(\RRR^n)$ denotes the space of Schwartz functions on $\RRR^n$ and $\mathcal{S}'(\RRR^n)$ its dual space, also called the \emph{space of tempered distributions}. Moreover, the Fourier transform $\widehat{\cdot}: \mathcal{S}(\RRR^n) \to \mathcal{S}(\RRR^n)$ and its inverse $\overset{\vee}{\cdot}: \mathcal{S}(\RRR^n) \to \mathcal{S}(\RRR^n)$ are defined for $f_1, f_2 \in \mathcal{S}(\RRR^n)$ by
\begin{align*}
\widehat{f_1}(\xi) &:= \int\limits_{\RRR^n} e^{-i2\pi x \xi } f_1(x) dx,\\
\overset{\vee}{f_2}(x) &:=\frac{1}{(2\pi)^n} \int\limits_{\RRR^n} e^{i2\pi x \xi } f_2(\xi) d\xi.
\end{align*}
We recall that $\, \widehat{\cdot}$ and $\overset{\vee}{\cdot}$ extend to mappings $\widehat{\cdot}: \mathcal{S'} \to \mathcal{S'}$ and $\overset{\vee}{\cdot}: \mathcal{S'} \to \mathcal{S'}$, see \cite{SimonReed}. \newline
For an open set $\Om\subset \RRR^n$ and $s \in \RRR$, we denote on $H^s(\Om)= W^{s,2}(\Om)$. \end{remark}


The following lemma follows from more general Sobolev embeddings for $W^{s,p}$ spaces, see for example \cite{Stein} and \cite{Adams} for details and proofs.
\begin{lemma}[Sobolev embeddings] \label{sobolevdomainlemma}
Let $\Om^2 \subset \RRR^2$ and $\Om^3 \subset \RRR^3$ be smooth domains. Then, the following are continuous embeddings,
\begin{align*}
&H^{1/2}(\Om^2) \hookrightarrow L^4(\Om^2), \\
&L^2(\Om^2) \hookrightarrow H^{-1/2}(\Om^2).
\end{align*}
If, furthermore, $\Om^2$ and $\Om^3$ are bounded, then the following continuous embeddings are also compact,
\begin{align*}
&H^2(\Om^3) \hookrightarrow C^{0,\a}(\ol{\Om^3}) \text{ for } \a \in (0,1/2), \\
&H^2(\Om^3) \hookrightarrow W^{1,4}(\Om^3), \\
&H^2(\Om^3) \hookrightarrow H^1(\Om^3), \\
&H^{3/2}(\Om^2) \hookrightarrow H^{1/2}(\Om^2).
\end{align*}
\end{lemma}


The following are standard product estimates, we refer the reader to Section 13.3 in \cite{TaylorPDE} and \cite{Maxwell} for more details and proofs.
\begin{lemma}[Product estimates] \label{lemmaMoserEstimates} Let $\Om^2 \subset \RRR^2$ be a smooth domain, and let $u$ and $v$ be functions on $\Om^2$. Then,
\begin{align*}
\Vert uv \Vert_{H^{1/2}(\Om^2)} \lesssim \Vert u \Vert_{H^{5/4}(\Om^2)} \Vert v \Vert_{H^{1/2}(\Om^2)}.
\end{align*}
Moreover, for every integer $m\geq0$, there is a constant $C_m>0$ such that
\begin{align*}
\Vert u v \Vert_{H^{m+3/2}(\Om^2)} \lesssim& \Vert u \Vert_{H^{m+3/2}(\Om^2)} \Vert v \Vert_{H^{3/2}(\Om^2)} + \Vert u \Vert_{H^{3/2}(\Om^2)} \Vert v \Vert_{H^{m+3/2}(\Om^2)} \\
&+ C_m ( \Vert u \Vert_{H^{3/2}(\Om^2)} + \Vert v \Vert_{H^{3/2}(\Om^2)}).
\end{align*}
Let $\Om^3 \subset \RRR^3$ be a smooth domain, and let $u$ and $v$ be functions on $\Om^3$. Then for every integer $m\geq0$ there is a constant $C_m>0$ such that
\begin{align*}
\Vert u v \Vert_{H^{m}(\Om^3)} \lesssim& \Vert u \Vert_{H^{m}(\Om^3)} \Vert v \Vert_{H^{2}(\Om^3)} + \Vert u \Vert_{H^{2}(\Om^3)} \Vert v \Vert_{H^{m}(\Om^3)} \\
&+ C_m ( \Vert u \Vert_{H^{2}(\Om^3)} + \Vert v \Vert_{H^{2}(\Om^3)}).
\end{align*}
\end{lemma}


We define now the \emph{trace operator} for continuous functions.
\begin{definition}[Trace operator for continuous functions] Let $n\geq1$ be an integer and let $\Om \subset \RRR^n$ be a smooth domain. Let $f$ be a continuous scalar function on $\ol{\Om}$. Denote the restriction of $f$ to $\pr \Om$ by 
\begin{align*}
\tau(f) := f \vert_{\pr \Om}.
\end{align*}
\end{definition}


\ni The trace operator $\tau$ extends to Sobolev spaces as follows.
\begin{lemma}[Trace operator for $W^{k,p}$-functions] \label{sobolevtracedomain} Let $n\geq1$ and $k\geq1$ be integers and $1<p< \infty$ a real. Let $\Om \subset \RRR^n$ be a smooth domain. Then the trace operator $\tau$ extends to a bounded linear operator between the following function spaces,
\begin{align*}
W^{k,p}(\Om) \to W^{k-1/p,p}(\pr \Om).
\end{align*}
\end{lemma}

\begin{remark} \label{remark453546} We cannot apply Lemmas \ref{sobolevdomainlemma}, \ref{lemmaMoserEstimates} and \ref{sobolevtracedomain} directly to the sets $B^+(x,r)$ because they are not smooth domains. Therefore, whenever these lemmas are invoked in Sections \ref{sec:setup}, \ref{sec:ExistenceL2} and \ref{sec:harmregularitymflds}, we tacitly apply the lemmas to the smooth domain $\Om_{x,r',r}$ between $B^+(x,r')$ and $B^+(x,r)$, see its definition in \eqref{defomegq3424145}, so that the estimates hold on $B^+(x,r')$ for a slightly smaller $r'<r$. \end{remark}



In Section \ref{SectionTraceEstimateForK}, we work with the space $H^{-1/2}(\RRR^2)$. To ease the presentation, we now introduce a Fourier operator $\langle \prd \rangle^{-1}$ and summarise its basic properties in Lemma \ref{propprd}.
\begin{definition} Let $f\in \SS(\RRR^3)$ be a scalar function. Define
\begin{align*}
\langle \prd \rangle^{-1} f(\cdot,x^3):= \left( (1+ (\xi^1)^2+(\xi^2)^2)^{-1/2} \widehat{f}(\xi^1, \xi^2,x^3) \right)^{\vee},
\end{align*} 
where the Fourier transform and its inverse are taken with the respect the variables $x^1,x^2$ only.
\end{definition}


The proof of the next lemma is left to the reader.
\begin{lemma}[Properties of $\langle \prd \rangle^{-1}$] \label{propprd} The following holds.
\begin{itemize}
\item The operator $\langle \prd \rangle^{-1}$ extends to a mapping $L^2(\RRR^3) \to L^2(\RRR^3)$, and for two functions $f,f' \in L^2(\RRR^3)$, we have
\begin{align*}
\int\limits_{\RRR^3} f \left( \langle \prd \rangle^{-1} f' \right)= \int\limits_{\RRR^3} \left( \langle \prd \rangle^{-1} f \right) f'.
\end{align*}
\item Let $f \in \mathcal{S}(\RRR^3)$ be a scalar function and $s \in \RRR$,
\begin{align*}
\Vert \langle\prd \rangle^{-1} f(\cdot,x^3) \Vert_{H^{s+1}(\RRR^2)} = \Vert f(\cdot,x^3) \Vert_{H^s(\RRR^2)}.
\end{align*}
\item Let $f \in \mathcal{S}(\RRR^3)$ be a scalar function, then
\begin{align*}
\Vert f(\cdot, x^3)\Vert_{H^{-1/2}(\RRR^2)}^2 &:= \Vert (1+ \vert \cdot \vert^2)^{-1/4} \widehat{f}(\cdot,x^3) \Vert^2_{L^2(\RRR^2)}\\
& =  \int\limits_{\RRR^2} f(x^1,x^2,x^3) \langle \prd \rangle^{-1} f(x^1,x^2,x^3) dx^1 dx^2 .
\end{align*}
\item Let $f \in \mathcal{S}(\RRR^3)$ be a scalar function, then
\begin{align*}
\left[ \pr_{x^3} , \langle \prd \rangle^{-1} \right] f =0.
\end{align*}
\end{itemize}
\end{lemma}



\subsubsection{Tensor spaces}

\begin{definition}[Tensor spaces] Let $n\geq1$ be an integer. Let $\Om \subset \RRR^n$ be an open set and let $T$ be a tensor on $\Om$. For reals $s \in \RRR$ and $1<p<\infty$, and integers $k\geq0$, we let
\begin{align*}
\WW^{s,p}(\Om), \HH^{s}(\Om), \LL^p(\Om) \text{ and } \CC^{k}(\Om)
\end{align*}
denote the spaces of tensors whose coordinate components are respectively in 
\begin{align*}
W^{s,p}(\Om), H^{s}(\Om), L^p(\Om) \text{ and } C^k(\Om),
\end{align*}
equipped with the standard norm, that is, for example, for a $(l,m)$-tensor $T$ on $\Om$,
\begin{align*}
\Vert T \Vert_{\WW^{s,p}(\Om)} := \sum\limits_{i_1, \dots i_l=1}^n \sum\limits_{j_1 \dots j_m=1}^n \Vert T^{i_1 \dots i_{l}}_{\,\,\,\,\,\,\,\,\,\,\,\,\,\, j_1 \dots j_m} \Vert_{W^{s,p}(\Om)},
\end{align*}
where $T^{i_1 \dots i_{l}}_{\,\,\,\,\,\,\,\,\,\,\,\,\,\, j_1 \dots j_m}$ denotes the coordinate components.
\end{definition}


\subsection{Riemannian geometry and boundary harmonic coordinates} \label{sec:riemdg}


\begin{definition}[Volume radius at scale $r$] \label{def:volumeradius}
Let $(M,g)$ be a Riemannian $3$-manifold with boundary. For a real $r>0$ and a point $p \in M$, the \emph{volume radius at scale $r$ at $p$} is defined as
\begin{align*}
r_{vol}(r,p) := \inf\limits_{r'< r} \frac{\mathrm{vol}_g \left( B_g(p,r') \right)}{\frac{4\pi}{3} (r')^{3}},
\end{align*}
where $B_g(p,r')$ denotes the geodesic ball of radius $r'$ centered at $p$. The \emph{volume radius of $(M,g)$ at scale $r$} is defined as
\begin{align*}
r_{vol}(M,r) := \inf\limits_{p \in M} r_{vol}(r,p).
\end{align*}
\end{definition}


\begin{definition}[Boundary harmonic coordinates] \label{def:bdryharm}
Let $(M,g)$ be a smooth Riemannian $3$-manifold with boundary. For a point $x \in \HHH^+$ and a real $r>0$, let 
$\varphi: B^+(x,r) \to U \subset M$ 
be a chart of $M$ such that $\{x^3=0\} \subset \pr M$. The coordinates $(x^1,x^2,x^3)$ are called \emph{boundary harmonic} if
\begin{align*}
\triangle_g x^j &=0 \text{ on } B^+(x,r) \text{ for } j=1,2,3, \\
\Ld_\gd x^A &= 0 \text{ on } \ubp(x,r) \text{ for } A=1,2.
\end{align*}
Here $\Ld_\gd$ denotes Laplace-Beltrami operator of the induced metric $\gd$ on $\ubp(x,r)$. In this case, we also call the chart $\varphi: B^+(x,r) \to U \subset M$ boundary harmonic.
\end{definition}


\begin{remark} In the above definition, as in the rest of this paper, we abuse notation by not explicitly denoting the pullback of $g$ by $\varphi$. For example, we write $\triangle_g x^j $ on $B^+(x,r)$ instead of $\triangle_{\varphi^\ast g} x^j$ on $B^+(x,r)$. \end{remark}


The proof of the following properties of boundary harmonic coordinates is left to the reader, see also Lemma 11.2.6 in \cite{Petersen} for the case of manifolds without boundary.

\begin{lemma} \label{lemmaharmexpr}
Let $g$ be a Riemannian metric in boundary harmonic coordinates $(x^1, x^2, x^3)$ on $B^+(x,r)$. Then the following holds for $i,j=1,2,3$ and $A,B =1,2$.
\begin{itemize}
\item The coordinate components $g_{ij}$ satisfy on $B^+(x,r)$
\begin{align*}
\half \triangle_g g_{ij} + Q_{ij} = -\RRRic_{ij},
\end{align*}
where $Q_{ij}(g,\pr g) := -  g \left( \nab g_{ik}, \nab g^{kl} \right) g_{lj} +  g\left( \mathrm{Hess} \, x^k, \mathrm{Hess} \, x^l \right) g_{ik} g_{lj}$ and $\RRRic$ is the Ricci tensor of $g$. 
\item The components $g^{ij}$ of the inverse metric satisfy on $B^+(x,r)$
\begin{align*}
\half \triangle_g g^{ij} - Q^{ij} = \RRRic^{ij},
\end{align*}
where $Q^{ij} := g\left( \mathrm{Hess} \, x^i, \mathrm{Hess} \, x^j \right)$, and $\RRRic^{ij} = g^{il} g^{jm} \RRRic_{lm}$.
\item The coordinate components of the induced metric $\gd_{AB}$ satisfy on $\ubp(x,r)$
\begin{align*}
\half \Ld_\gd \gd_{AB} + \Qdd_{AB} &= - K \gd_{AB},
\end{align*}
where $\Qdd_{AB}:= - \gd \left( \Nd \gd_{AC}, \Nd \gd^{CD} \right) \gd_{DB} + \gd_{AC} \gd\left(  {\mathrm{Hess} \mkern-15mu / \mkern+5mu} \, x^C,  {\mathrm{Hess} \mkern-15mu / \mkern+5mu} \, x^D \right) \gd_{lB}$, with $\Nd$ the induced covariant derivative on $\ubp(x,r)$, and $K$ denotes the Gauss curvature of $\ubp(x,r)$.
\item The Laplace-Beltrami operators on $B^+(x,r)$ and $\ubp(x,r)$ are respectively given by
\begin{align*}
\triangle_g u = g^{ij} \pr_i \pr_j u, \, \Ld_\gd u = \gd^{AB} \prd_A \prd_B u.
\end{align*}
\end{itemize}
\end{lemma}

The quadratic non-linearities $\Qdd_{AB}\approx \prd \gd \prd \gd$ and $Q_{ij}, Q^{ij} \approx \pr g \pr g$ satisfy the following properties. The proof follows by standard Sobolev embeddings and product estimates (see Lemmas \ref{sobolevdomainlemma} and \ref{lemmaMoserEstimates}), and is left to the reader.
\begin{lemma} \label{LemmaQConvergenceEstimate} \label{Qcalcul} Let $\Om^2 \subset \RRR^2$ and $\Om^3 \subset \RRR^3$ be bounded smooth domains. Then the following holds.
\begin{itemize}
\item Let $(\gd_n)_{n \geq0}$ and $\gd$ be smooth Riemannian metrics on $\Om^2$. If
\begin{align*}
\gd_n \rightharpoonup \gd \text{ in } H^{3/2}(\Om^2) \text{ as } n \to \infty,
\end{align*}
then
\begin{align*}
\Vert \Qdd_{AB}(\gd_n) - \Qdd_{AB}(\gd) \Vert_{H^{-1/2}(\Om^2)} \to 0 \text{ as } n \to \infty.
\end{align*}

\item Let $(g_n)_{n \geq0}$ and $g$ be smooth Riemannian metrics on $\Om^3$. If
\begin{align*}
g_n \rightharpoonup g \text{ in } H^{2}(\Om^3) \text{ as } n \to \infty,
\end{align*}
then
\begin{align*}
\Vert Q_{AB}(g_n) - Q_{AB}(g) \Vert_{L^2(\Om^3)} \to 0 \text{ as } n \to \infty.
\end{align*}
\item There is $\varep>0$ small such that if 
\begin{align*}
\Vert \gd-e \Vert_{\HH^{3/2}(\Om^2)} < \varep,
\end{align*}
then for every integer $m>0$ there is a constant $C_m>0$ such that
\begin{align*}
\Vert \prd^{m-1} \Qdd_{AB} \Vert_{H^{1/2}(\Om^2)} \lesssim& \Vert \gd-e \Vert_{\HH^{3/2}(\Om^2)} \Vert \prd^{m+1} \gd \Vert_{\HH^{1/2}(\Om^2)} + \Vert \gd \Vert_{\HH^{m+1/2}(\Om^2)} \\
&+ C_m \Vert \gd-e \Vert_{\HH^{3/2}(\Om^2)}.
\end{align*}
\item There is $\varep>0$ small such that if 
\begin{align*}
\Vert g-e \Vert_{\HH^{2}(\Om^3)} < \varep,
\end{align*}
then for every integer $m>0$ there is a constant $C_m>0$ such that
\begin{align*}
\Vert \pr^{m} Q_{AB} \Vert_{L^2(\Om^3)} \lesssim& \Vert g-e \Vert_{\HH^{2}(\Om^3)} \Vert \pr^{m+2} g \Vert_{\LL^2(\Om^3)} + \Vert g \Vert_{\HH^{m+1}(\Om^3)} \\
&+ C_m \Vert g-e \Vert_{\HH^{2}(\Om^3)}.
\end{align*}
\end{itemize}
\end{lemma}


In boundary harmonic coordinates $(x^1,x^2,x^3)$ on $B^+(x,r)$, we can express $g$ as
\begin{align} \label{eaexpresseion98976}
g = a^2 \left(dx^3\right)^2 + \gd_{AB} \left( \be^A dx^3 + dx^A \right) \left( \be^B dx^3 + dx^B \right).
\end{align}
with
\begin{itemize}
\item the lapse function $a>0$, 
\item the induced Riemannian metric $\gd$ on level sets $\{ x^3 = \mathrm{const.} \}$,
\item the shift vector $\be$, tangent to $\{ x^3 = \mathrm{const.}\}$.
\end{itemize}
The outward-pointing\footnote{The outward-pointing unit normal has by definition negative $\pr_{x^3}$-component.} unit normal to the level sets $\{ x^3 = \mathrm{const.} \}$ is given by
\begin{align} \label{unitnormal}
N := -\frac{1}{a} \pr_3 + \frac{1}{a} \be^A \pr_A = - a \nab x^3.
\end{align}
It holds that $\det g = a^2 \det \gd$ and the components $g^{ij}$ of the metric inverse are given by
\begin{align} \label{invcompexpr143}
g^{AB} = \gd^{AB} + \frac{\be^A \be^B}{ a^2},
g^{A3} = -\frac{\be^A}{a^2}, 
g^{33} = \frac{1}{a^2}.
\end{align}
The second fundamental form\footnote{Here we use the sign convention $\Theta(X,Y) := - g(X,\nab_Y N)$.} on $\ubp(x,r)$ is given in general coordinates by
\begin{align} \label{secondformexpr3}
\Th_{AB}= \frac{1}{2a} \pr_3 (g_{AB})  - \frac{1}{2a}\left( \Lied_{\be} \gd \right)_{AB},
\end{align}
where $\Lied$ denotes the induced Lie derivative on $\ubp(x,r)$.

\begin{lemma} \label{NeumannHarmonic} Let $g$ be a Riemannian metric in boundary harmonic coordinates on $B^+(x,r)$. Let $N$ and $\Th$ denote the outward-pointing unit normal and the second fundamental form of $\ubp(x,r)$, respectively. Then it holds that on $\ubp(x,r)$
\begin{align*} 
N \left( g^{33} \right) &= 2 \tr \Th g^{33},  \\
N \left( g^{3A} \right) &=  \tr \Th g^{3A} - \half \frac{1}{\sqrt{g^{33}}} g^{Am} \pr_m g^{33}. 
\end{align*}
\end{lemma}

\begin{proof} In general, for smooth functions $f$,
\begin{align*}
\triangle_g f = N(N(f)) - \tr \Th N(f) + \Ld_\gd f + a^{-1} g(\Nd a, \Nd f).
\end{align*}

By assumption we have for $j=1,2,3$ that $\triangle x^j =0$ and $\Ld x^j =0$. Hence
\begin{align} \label{essumcons542}
0 = \triangle x^j = N(N(x^j)) - \tr \Th N(x^j) + a^{-1} g(\Nd a, \Nd x^j).
\end{align}
By \eqref{unitnormal} and \eqref{invcompexpr143} it follows for $A=1,2$ that
\begin{align} \begin{aligned}
N(N(x^3)) &= -N(\sqrt{g^{33}}) \\
&= -\half a N(g^{33}), \\
N(N(x^A)) &= -N\left(a g^{3A} \right) \\
&= -N(a) g^{3A}-a N(g^{3A})\\
&= -N(a) g(\nab x^3, \nab x^A) - a N(g^{3A}) \\
&= a^{-1} g(\nab x^A, N) N(a) - a N(g^{3A}),
\end{aligned} \label{star1} \end{align}
and moreover
\begin{align}
\tr \Th N(x^j) = -\tr \Th a g^{3j}. \label{starstar1}
\end{align}

By \eqref{star1}, \eqref{starstar1}, \eqref{essumcons542} and using that $\Nd x^3 = 0$, we get
\begin{align*}
N(g^{33}) =& 2 \tr \Th g^{33}, \\
N(g^{3A}) =&  a^{-2} \left(  g(\nab x^A, N) N(a)+ g( \Nd x^A, \Nd a) \right) + \tr \Th g^{3A} \\
=&  a^{-2} (\nab x^A)(a) + \tr \Th g^{3A} \\
=& -\frac{1}{2\sqrt{g^{33}}} (\nab x^A)(g^{33})+ \tr \Th g^{3A}.
\end{align*}
This finishes the proof of Lemma \ref{NeumannHarmonic}. \end{proof}


The next lemma shows that bounds on the metric components $g^{33},g^{3A},g_{AB}$ imply bounds for all metric components $g_{ij}$.
\begin{lemma}[Control of all metric components] \label{upstairscontrol}
Let $\Om \subset \RRR^3$ be a smooth domain and let $g$ be a smooth Riemannian metric on $\Om$. There is $\varep>0$ such that if
\begin{align*}
\Vert g-e \Vert_{\HH^2(\Om)} < \varep,
\end{align*}
then for every integer $m\geq0$, there is a constant $C_{m}>0$ such that
\begin{align} \label{eqfullcontrol13242}
\Vert g \Vert_{\HH^m(\Om)} \lesssim  \sum\limits_{A,B=1,2} ( \Vert g_{AB} \Vert_{\HH^m(\Om)} +\Vert g^{33} \Vert_{\HH^m(\Om)} + \Vert g^{3A} \Vert_{\HH^m(\Om)} )+ C_{m} \Vert g-e \Vert_{\HH^2(\Om)}.
\end{align}
\end{lemma}


\begin{proof} It suffices to control the components $g_{3A}$ and $g_{33}$. By \eqref{eaexpresseion98976}, \eqref{invcompexpr143}, and the general property $g_{3k} g^{k3} = 1$,
\begin{align*}
g_{3A} = g_{AC} \frac{g^{3C}}{g^{33}}, \, \, g_{33} = \frac{1}{g^{33}}\left(1- g_{3A}g^{3A} \right).
\end{align*}
From this, the product estimates of Lemma \ref{lemmaMoserEstimates} imply \eqref{eqfullcontrol13242}. Details are left to the reader.
 \end{proof}


\subsection{The trace estimate for the Gauss curvature $K$} \label{SectionTraceEstimateForK}

Let $g$ be a Riemannian metric in boundary harmonic coordinates on $B^+(x,r)$ such that 
\begin{align*}
\Vert g-e \Vert_{\HH^2(B^+(x,r))} < \varep
\end{align*}
for a small constant $\varep>0$.

In this section, we prove that if $\varep>0$ is sufficiently small, then for all $0<r'<r$ it holds that
\begin{align} \label{Kestimate}
\Vert K \Vert_{H^{-1/2}(\ubp(x,r'))} \leq C_{r',r} \Big( \Vert \RRRic \Vert_{L^2(B^+(x,r))} + \Vert \Th \Vert_{{L^4}(\ubp(x,r))} \Big).
\end{align}
We remark that such a trace estimate was already proved in \cite{KlRod} in Besov spaces, see also \cite{J3}. For the convenience of the reader, we give here a proof of \eqref{Kestimate} without Besov spaces.

By the twice contracted Gauss equation 
\begin{align*}
2K = (\tr \Th)^2 - \vert \Th \vert^2 + \mathrm{R}_{scal} - 2 \RRRic(N,N),
\end{align*}
and Lemmas \ref{sobolevdomainlemma} and \ref{sobolevtracedomain} it follows that for $\varep>0$ sufficiently small,
\begin{align*}
\Vert K \Vert_{H^{-1/2}(\ubp(x,r'))} \lesssim \Vert \mathrm{R}_{scal} - 2 \RRRic(N,N) \Vert_{H^{-1/2}(\ubp(x,r'))} + \Vert \Th \Vert_{{L^4}(\ubp(x,r))}.
\end{align*}

Therefore, the proof of \eqref{Kestimate} follows from the next proposition.
\begin{proposition}[Trace estimate] \label{PropositionTraceEstimate}
Let $g$ be a Riemannian metric in boundary harmonic coordinates on $B^+(x,r)$. There is an $\varep>0$ small such that if
\begin{align*}
\Vert g-e \Vert_{\HH^2(B^+(x,r))} < \varep,
\end{align*}
then it holds that for all $0<r'<r$,
\begin{align*}
\Vert \mathrm{R}_{scal} - 2 \RRRic(N,N) \Vert_{H^{-1/2}(\ubp(x,r'))} \leq C_{r',r} \Vert \RRRic \Vert_{L^2(B^+(x,r))}.
\end{align*}
\end{proposition}

In the rest of this section, we prove Proposition \ref{PropositionTraceEstimate}. The proof of this trace estimate is based on the identity
\begin{align} \begin{aligned}
&N \left( \mathrm{R}_{scal}-2\RRRic_{NN} \right) \\
=& \nab_N \mathrm{R}_{scal} -2 \nab_N \RRRic_{NN} + 4 \RRRic_{AN} (\nab_N N)^A \\
=& 2 \nab^A \RRRic_{AN} + 4 \RRRic_{AN} (\nab_N N)^A \\
=& 2 \Divd \left( \RRRic_{\cdot \,\, N} \right) - 2 \tr \Th \RRRic_{NN} +2 \Th^{AB} \RRRic_{AB}+ 4 \RRRic_{AN} (\nab_N N)^A.
\end{aligned} \label{deridentity} 
\end{align}

We note that \eqref{deridentity} follows from the twice contracted Bianchi identity
\begin{align*}
\nab^l \RRRic_{lj} = \half \nab_l \mathrm{R}_{scal}.
\end{align*}

Let $\chi := \chi_{x,r',r}$ be a smooth cut-off function from $B^+(x,r')$ to $B^+(x,r)$ as defined in \eqref{defchi}. By Lemma \ref{propprd} and the fundamental theorem of calculus,
\begin{align*} 
&\Vert \mathrm{R}_{scal}-2\RRRic_{NN} \Vert^2_{H^{-1/2}\left(\ubp(x,r')\right)} \\
\leq& \Vert \chi^2 (\mathrm{R}_{scal}-2\RRRic_{NN})\Vert_{H^{-1/2}(\RRR^2)}^2 \\
=& \int\limits_{\{x^3 = 0\} } \chi^2 (\mathrm{R}_{scal}-2\RRRic_{NN}) \langle \prd \rangle^{-1} (\chi^2 (\mathrm{R}_{scal}-2\RRRic_{NN})) dx^1dx^2 \\
=&\int\limits_{0}^r \pr_{3} \left( \, \int\limits_{\RRR^2} \chi^2 (\mathrm{R}_{scal}-2\RRRic_{NN}) \langle \prd \rangle^{-1} (\chi^2(\mathrm{R}_{scal}-2\RRRic_{NN})) dx^1dx^2  \right)dx^3 \\
=& 2 \int\limits_{0}^r   \int\limits_{\RRR^2} \chi^2 (\mathrm{R}_{scal}-2\RRRic_{NN}) \langle \prd \rangle^{-1} \left( \pr_{3} \Big( \chi^2 (\mathrm{R}_{scal}-2\RRRic_{NN}) \Big) \right) ,
\end{align*}

Expressing $\pr_{3} = a N+ \be$, see \eqref{unitnormal}, we get
\begin{align}\begin{aligned}
&\Vert \mathrm{R}_{scal}-2\RRRic_{NN} \Vert^2_{H^{-1/2}\left(\ubp(x,r')\right)} \\
\leq&  \int\limits_{0}^r  \int\limits_{\RRR^2}\chi^2(\mathrm{R}_{scal}-2\RRRic_{NN})  \langle \prd \rangle^{-1}\Big(\chi^2aN  (\Rscal-2\RRRic_{NN})  \Big) \\
&+ \int\limits_{0}^r  \int\limits_{\RRR^2} \chi^2(\mathrm{R}_{scal}-2\RRRic_{NN}) \langle \prd \rangle^{-1}\Big(\chi^2 \be (\Rscal-2\RRRic_{NN})  \Big)  \\
&+\int\limits_{0}^r  \int\limits_{\RRR^2} \chi^2(\mathrm{R}_{scal}-2\RRRic_{NN})  \langle \prd \rangle^{-1} \Big((\pr_{x^3}\chi^2)  (\mathrm{R}_{scal}-2\RRRic_{NN}) \Big) .
 \end{aligned} \label{Kconveq1} \end{align}

The first term on the right-hand side of \eqref{Kconveq1} equals by \eqref{deridentity},
\begin{align} \begin{aligned} 
&\int\limits_{0}^r  \int\limits_{\RRR^2} \chi^2 (\mathrm{R}_{scal}-2\RRRic_{NN})\langle \prd \rangle^{-1}\Big(\chi^2 a N  (\mathrm{R}_{scal}-2\RRRic_{NN})  \Big)   \\
=&  2\int\limits_{0}^r \int\limits_{\RRR^2}\chi^2 (\mathrm{R}_{scal}-2\RRRic_{NN}) \langle \prd \rangle^{-1}\Big( \chi^2 a \Divd R_{\cdot \,\, N}  \Big)   \\
&-2 \int\limits_{0}^r  \int\limits_{\RRR^2}\chi^2 (\mathrm{R}_{scal}-2\RRRic_{NN}) \langle \prd \rangle^{-1}\Big( \chi^2 a (\tr \Th \RRRic_{NN} + \Th^{AB} \RRRic_{AB} + \RRRic_{AN} (\nab_N N)^A) \Big) .
\end{aligned} \label{ImportantEQ2322}
\end{align}

The first term on the right-hand side of \eqref{ImportantEQ2322} is estimated by using that by Lemma \ref{propprd}, the definition $\Divd X = \gd^{ij} \prd_i X_j$ for vectorfields $X$ and standard product estimates, see Lemma \ref{lemmaMoserEstimates}, we have for $\varep>0$ sufficiently small,
\begin{align*} 
 \left\Vert \langle \prd \rangle^{-1}\Big( \chi^2 a \Divd \RRRic_{\cdot \,\, N}  \Big) \right\Vert_{L^2(\RRR^3)} \leq  C_{r',r} \Vert \chi \RRRic \Vert_{L^2(\RRR^2)},
\end{align*}
and therefore
\begin{align*}
&2\int\limits_{0}^r \int\limits_{\RRR^2}\chi^2 (\mathrm{R}_{scal}-2\RRRic_{NN}) \langle \prd \rangle^{-1}\Big( \chi^2 a \Divd R_{\cdot \,\, N}  \Big)   \\
\lesssim& 2 \int\limits_0^r \Vert \chi^2 (\Rscal - 2 \RRRic_{NN}) \Vert_{L^2(\RRR^2)} \Vert \chi \RRRic \Vert_{L^2(\RRR^2)} dx^3 \\
\leq & C_{r',r} \Vert \RRRic \Vert^2_{L^2(B^+(x,r))}.
\end{align*}

The other terms are also estimated by Lemmas \ref{propprd} and \ref{lemmaMoserEstimates}. Indeed, the second term of \eqref{Kconveq1} involves only tangential derivatives analogously as the first term in \eqref{ImportantEQ2322}, and the second term of \eqref{ImportantEQ2322} and third term of \eqref{Kconveq1} are lower order. Details are left to the reader. We conclude that for $\varep>0$ sufficiently small,
\begin{align*}
\Vert \mathrm{R}_{scal}-2\RRRic_{NN} \Vert^2_{H^{-1/2}\left(\ubp(x,r')\right)} \leq& C_{r',r} \Vert \RRRic \Vert^2_{L^2(B^+(x,r))}.
\end{align*}
This finishes the proof of Proposition \ref{PropositionTraceEstimate}, and hence the proof of  \eqref{Kestimate}.


\section{The precise version of the main result} \label{ExactStatement} The following is our main result.

\begin{theorem}[Existence of regular coordinates, version 2] \label{thm:higherintroexharm}
Let $(M,g)$ be a compact, complete Riemannian $3$-manifold with boundary such that 
\begin{align*}
&\Vert \RRRic \Vert_{ L^2(M)}< \infty, \, r_{vol}(M,1) > 0, \, \Vert \Th \Vert_{ L^4(\pr M)}<\infty.
\end{align*}
Then the following holds. 
\begin{enumerate} 
\item {\bf $L^2$ regularity.} There is $\varep_0>0$ such that for every $0< \varep< \varep_0$, there is a radius
\begin{align*}
r=r( \Vert \RRRic \Vert_{ L^2(M) },\Vert \Th \Vert_{ L^4(\pr M)},r_{vol}(M,1),  \varep) >0
\end{align*}
such that for every point $p \in M$, there is a boundary harmonic chart\footnote{That is, the coordinates $x^i:= \left( \varphi^{-1}\right)^i: U \to \RRR$ are boundary harmonic.}
\begin{align*}
\varphi: B^+(x,r) \to U \subset M
\end{align*} 
with $\varphi(x)=p$ in which, on $B^+(x,r)$,
\begin{align*}
(1-\varep) e_{ij} \leq g_{ij} &\leq (1+\varep) e_{ij},
\end{align*}
where $e_{ij}$ denotes the Euclidean metric. Moreover,
\begin{align*}
r^{-1/2} \Vert \pr g \Vert_{\LL^2(B^+(x,r))} + r^{1/2}\Vert \pr^2 g \Vert_{\LL^2(B^+(x,r))} <& \varep.
\end{align*}
\item {\bf Higher regularity.} Let $m\geq1$ be an integer. Assuming higher regularity of $\RRRic$, we further have the higher regularity estimate
\begin{align*}
\Vert g \Vert_{\HH^{m+2}(B^+(x,r))} \leq C_{r} \sum\limits_{i=0}^m \Vert \nab^{(i)} \RRRic \Vert_{L^2(M)}  + C_{m,r} \varep.
\end{align*}
\end{enumerate}
\end{theorem}

\begin{remark} Cheeger-Gromov convergence theory for manifolds with boundary has been studied before, for example in \cite{Kodani}, \cite{Anderson2}, \cite{Anderson} and \cite{Perales}, see Section \ref{secCheegerGromovhistory}. These results assume pointwise bounds on $\RRRic$, while in Theorem \ref{thm:higherintroexharm} we only assume $L^2$-bounds on $\RRRic$.
\end{remark}

The proof of Theorem \ref{thm:main1} is split into two. We prove part (1) in Theorem \ref{thm:main1}, and part (2) in Proposition \ref{thm:main1Higherreg}.


\section{Convergence of Riemannian manifolds with boundary} \label{sec:setup}

In this section, we set up the Cheeger-Gromov theory on manifolds with boundary in low regularity. In the case of manifolds without boundary, these results are available for example in \cite{Petersen} and \cite{PetersenBook}.


\subsection{Convergence of functions, tensors and Riemannian manifolds} In this section, we introduce basic definitions of convergence.

\begin{definition}[$C^{m,\a}$- and $H^2$-convergence of functions and tensors] Let $m\geq0$ be an integer and $\a \in (0,1)$ a real. Let $(M,g)$ be a Riemannian manifold with boundary. Let $A\subset M$ be a pre-compact subset and $(U_n, \varphi_n)$ a finite number of fixed charts covering $A$. A sequence of functions $(f_i)_{i\in \mathbb{N}}$ on $A$ is said to converge in $C^{m,\alpha}$ (in $H^2$) as $i \to \infty$, if for each $n$, the pullbacks $(\varphi_n)^*f_i$ converge in $C^{m,\alpha}$ (in $H^2$) as $i\to \infty$. The convergence of a sequence of tensors on $A$ in $C^{m,\a}$ (in $H^2$) is defined similarly. \end{definition}


The so-called \emph{pointed convergence} of manifolds with boundary in $C^{m,\a}$ and $H^2$ is defined as follows.
\begin{definition}[Pointed $C^{m,\a}$- and $H^2$-convergence of manifolds with boundary] \label{def:pointedconv1}
A sequence $(M_i,g_i,p_i)$ of pointed Riemannian manifolds with boundary is said to converge to a pointed Riemannian manifold with boundary $(M,g,p)$ in the \emph{pointed $C^{m,\a}$-topology (in the pointed $H^2$-topology)} as $i \to \infty$, if for all $R>0$ there exists a bounded set $\Om \subset M$ with smooth boundary containing the geodesic ball $B_g(p,R) \subset \Om$ and embeddings $F_i: \Om \to M_i$ for large $i$ such that $F_i(p)=p_i$, $B_g(p_i,R) \subset F_i(\Om)$ and $(F_i)^\ast g_i \to g$ in the $C^{m,\a}$-topology (in the $H^2$-topology) on $\Om$ as $i\to \infty$.
\end{definition}

\begin{remark}
In this paper we consider sequences $(M_i,g_i)$ of smooth Riemannian manifolds with boundary that have uniform quantitative bounds, see for example the assumptions of Theorem \ref{thm:main1}. These sequences are then shown to convergence to a limit $(M,g)$ in a topology suited to the uniform bounds. The a priori regularity of the limit manifold is determined by the strength of the topology in which the convergence takes place.
\end{remark}


\subsection{The $C^{m, \alpha}$-norm of a Riemannian manifold} \label{secCma}

In this section, we define the $C^{m,\a}$-norm of a Riemannian manifold with boundary and state a compactness result for this norm.
\begin{definition} Let $m \geq0$ be an integer and $\a \in (0,1)$ a real. Let $(M,g,p)$ be a pointed smooth Riemannian $3$-manifold with boundary. For two given reals $r>0$ and $Q\geq0$, we write
\begin{align*} 
\Vert (M,g,p) \Vert_{C^{m,\a},r} \leq Q
\end{align*} 
if there exist a point $x \in \mathbb{H}^+$ and a $C^{m+1,\a}$-regular chart $\varphi: B^+(x,r) \to U \subset M$ with $\varphi(x)=p$ such that
\begin{itemize}
\item[(n1)] $\vert D\varphi \vert \leq e^Q$ on $B^+(x,r)$ and $\vert D \varphi^{-1} \vert \leq e^Q$ on $U$.\footnote{Here we denote \begin{align*}
\vert D \varphi\vert := \sup\limits_{\xi \in T\RRR^3, \vert \xi \vert =1} \sqrt{g(D\varphi(\xi), D\varphi(\xi))}, \,
\vert D \varphi^{-1} \vert := \sup\limits_{X \in TM, \vert X \vert_g =1} \vert D \varphi^{-1}(X) \vert.
\end{align*}}
\item[(n2)] For all multi-indices $I$ with $0 \leq \vert I \vert \leq m$, 
\begin{align*}
\max\limits_{i,j=1,2,3} r^{ \vert I \vert + \a} \left[ \pr^I g_{ij} \right]_{C^{0,\a}(B^+(x,r))} \leq Q,
\end{align*}
where $g_{ ij}$ denotes the coordinate components of $g$ in the chart $\varphi$.
\end{itemize}
We set 
$$\Vert (M,g,p) \Vert_{C^{m,\a},r} := \inf \{ Q\geq0: \Vert (M,g,p) \Vert_{C^{m,\a},r} \leq Q
\},$$
and let $\Vert (M,g,p) \Vert_{C^{m,\a},r}= +\infty$ if the infimum does not exist. Define further
$$\Vert (M,g) \Vert_{C^{m,\a},r} := \sup\limits_{p \in M} \Vert (M,g,p) \Vert_{C^{m,\a},r}.$$
\end{definition}


\begin{remark} Let $(M,g,p)$ be a compact pointed smooth Riemannian manifold with boundary. Then, the $C^{m,\a}$-norm is finite for all $r>0$, that is, $$\Vert (M,g,p)\Vert_{C^{m,\a},r}< \infty.$$
\end{remark}


\begin{remark} The above condition (n1) is equivalent to
\begin{align} \label{linftycontrol}
e^{-2Q} e_{ij} \leq g_{ij} \leq e^{2Q} e_{ij},
\end{align}
where $e$ denotes the Euclidean metric.
\end{remark}


We have the following properties of the $C^{m,\a}$-norm.
\begin{proposition} \label{prop:normprops}
Let $m\geq0$ be an integer, $\a \in (0,1)$ a real and let $(M,g,p)$ be a pointed smooth Riemannian $3$-manifold with boundary. Then the following holds.
\begin{enumerate}
\item $\Vert (M,g,p) \Vert_{C^{m,\a},r} = \Vert (M,\la^2 g,p) \Vert_{C^{m,\a},\la r}$ for all $\la >0$.
\item The function $r\mapsto \Vert (M,g,p) \Vert_{C^{m,\a},r}$ is increasing, continuous and converges to $0$ as $r\to 0$.
\item If $\Vert (M,g,p) \Vert_{C^{m,\a},r} < Q$, then there exists a chart $(\varphi,U)$ such that the geodesic ball $B_g(p,e^{-Q}r ) \subset U$.
\end{enumerate}
\end{proposition}
\begin{proof} We omit the proofs of the parts (1) and (2), because they are analogous to the proof of Proposition \ref{prop:harmnormprops} below. Point (3) follows by using that with \eqref{linftycontrol}, the geodesic length in $M$ and the coordinate length in a chart can be compared.

For more details, see Proposition 11.3.1 in \cite{PetersenBook} for the case of manifolds without boundary. \end{proof}


We have the following compactness result.
\begin{theorem}[Fundamental theorem of manifold convergence theory] \label{thm:fundamental}
For an integer $m\geq0$ and reals $Q>0$, $\a \in (0,1)$ and $r>0$, let $\mathcal{M}^{m,\a}(Q,r)$ denote the class of pointed smooth Riemannian $3$-manifolds with boundary $(M,g,p)$ such that $$\Vert (M,g) \Vert_{C^{m,\a},r} \leq Q.$$ 
The set $\mathcal{M}^{m,\a}(Q,r)$ is compact in the pointed $C^{m,\beta}$-topology for all $\beta< \alpha$.
\end{theorem}
As already noted in Section 3.1 in \cite{Anderson}, the well-known proof for manifolds without boundary (see for example Theorem 11.3.6 in \cite{PetersenBook}) extends to manifolds with boundary without significant changes. We refer the reader to these literature results for a proof of Theorem \ref{thm:fundamental}.


\subsection{The boundary harmonic $H^2$-norm of a Riemannian manifold} 

In this section, we define the $H^2$-norm of a Riemannian manifold with boundary and state a pre-compactness result for this norm.

\begin{definition} \label{def:harmnorm}
Let $(M,g,p)$ be a pointed smooth Riemannian $3$-manifold with boundary. For two reals $r>0$ and $Q\geq0$, we write 
$$\Vert (M,g,p) \Vert^{b.h.}_{H^2,r} \leq Q,$$ 
if there exist a point $x \in \mathbb{H}^+$ and a $H^3$-regular chart $\varphi: B^+(x,r) \to U \subset M$ with $\varphi(x)=p$ such that
\begin{itemize}
\item[(h1)] $\vert D\varphi \vert \leq e^Q$ on $B^+(x,r)$ and $\vert D \varphi^{-1} \vert \leq e^Q$ on $U$. 
\item[(h2)] For all multi-indices $I$ with $1 \leq \vert I \vert \leq 2$, 
\begin{align*}
\max\limits_{i,j=1,2,3}r^{\vert I \vert- 3/2} \Vert \pr^I g_{ij} \Vert_{L^2(B^+(x,r))} \leq Q,
\end{align*}
where $g_{ij}$ denotes the coordinate components of $g$ in the chart $\varphi$. 
\item[(h3)] The coordinates $\varphi^{-1}: U \to B^+(x,r)$ are boundary harmonic.
\end{itemize}
We set 
\begin{align*}
\Vert (M,g,p) \Vert^{b.h.}_{H^2,r} := \inf \{ Q \geq0: \Vert (M,g,p) \Vert^{b.h.}_{H^2,r} \leq Q\}
\end{align*}
and let $\Vert (M,g,p) \Vert^{b.h.}_{H^2,r}= +\infty$ if the infimum does not exist. Let further
$$\Vert (M,g) \Vert^{b.h}_{H^2,r} := \sup\limits_{p \in M} \Vert (M,g,p) \Vert^{b.h.}_{H^2,r}.$$
\end{definition}

\begin{remark}
Let $(M,g,p)$ be a compact pointed smooth Riemannian manifold with boundary $(M,g,p)$. Then the $H^2$-norm is finite for all $r>0$, that is,
$$\Vert (M,g,p)\Vert^{b.h.}_{H^2,r}< \infty.$$
\end{remark}

\begin{remark} 
The above condition (h1) is equivalent to
\begin{align*}
e^{-2Q} e_{ij} \leq g_{ij} \leq e^{2Q} e_{ij},
\end{align*}
where $e$ denotes the Euclidean metric.
\end{remark}

We have the following properties of the boundary harmonic $H^2$-norm.
\begin{proposition} \label{prop:harmnormprops}
Let $r>0$ be a real and let $(M,g,p)$ be a smooth pointed Riemannian $3$-manifold with boundary and $p \in M$ a point. The following holds.
\begin{enumerate}
\item $\Vert (M, \la^2 g,p) \Vert^{b.h.}_{H^2, \la r} = \Vert (M,g,p) \Vert^{b.h.}_{H^2,r}$ for all reals $\la >0$.
\item The function $r \mapsto \Vert (M,g,p) \Vert^{b.h.}_{H^2,r}$ is increasing, continuous and converges to $0$ as $r\to 0$.
\item Let $(M_i,g_i,p_i)$ be a sequence of pointed smooth Riemannian manifolds with boundary and $(M,g,p)$ be a pointed smooth Riemannian manifold. If  
$$(M_i,g_i,p_i) \to (M,g,p)$$
 as $i\to \infty$ in the pointed $H^2$-topology, then it holds that
\begin{align*}
\Vert (M_i, g_i, p_i) \Vert^{b.h.}_{H^2,r} \to \Vert (M,g,p ) \Vert^{b.h.}_{H^2,r}.
\end{align*}
\end{enumerate}
\end{proposition}


\begin{proof} {\bf (1).} For a real $\la >0$, consider the scaled 
\begin{align*}
\varphi_\la(x):= \varphi (\la^{-1} x):B^+(\la x, \la r) \to U \subset M , \, g_\la := \la^2 g.
\end{align*}
This pair satisfies the conditions (h1) and (h2) on $B^+(\la x, \la r)$ in Definition \ref{def:harmnorm} with the same $Q< \infty$ as $(\varphi,g)$ on $B^+( x,  r)$. Indeed, this follows because on $B^+(\la^{-1} x, \la^{-1}r)$
$$(\varphi_\la)^\ast g_\la (x) = \varphi^\ast g (\la x).$$

It remains to show that if $\varphi^{-1}$ is boundary harmonic for $g$, then $\left( \varphi_\la\right)^{-1}$ is boundary harmonic for $g_\la$. This follows by the invariance of the Laplace-Beltrami operator and is left to the reader. This finishes the proof of point (1).\\

\noindent {\bf (2).} The monotonicity of the norm in $r>0$ follows directly by restriction. The continuity follows from the above scaling invariance by estimating for $\tilde{r}>0$
\begin{align} \label{continuityinr3434}
\Vert (M,g) \Vert^{b.h.}_{H^2,\tilde{r}} = \left\Vert \left(M,\left( \frac{r}{\tilde{r}} \right)^{-2}g\right) \right\Vert^{b.h.}_{H^2,r} \leq \max \left\{ Q + \left\vert \log \left( \frac{\tilde{r}}{r} \right) \right\vert,  \left( \frac{\tilde{r}}{r}\right)^2 Q\right\}.
\end{align}
The estimate \eqref{continuityinr3434} implies that the norm is continuous in $r$, for more details see Propositions 11.3.2 and 11.3.5 in \cite{PetersenBook} for the case of manifolds without boundary.

It remains to show that as $r\to 0$,
$$\Vert (M,g,p) \Vert^{b.h.}_{H^2,r} \to 0.$$
However, by Proposition \ref{prop:normprops} we have that as $r\to 0$,
\begin{align*}
\Vert (M,g,p) \Vert_{C^{2,1/2},r} \to 0.
\end{align*}
Approximate these coordinates by boundary harmonic coordinates with appropriate prescribed Dirichlet and Neumann boundary data on the boundary, see Lemmas \ref{lemmaharmexpr} and \ref{NeumannHarmonic}. The standard elliptic estimates in Corollaries \ref{prop:ellest3} and \ref{fullldg6858} in the appendix show that these boundary harmonic coordinates admit the right bounds of size $\varep$. Details are left to the reader. \\


\noindent {\bf (3).} It suffices to prove the following two inequalities,
\begin{align}
\limsup\limits_{i\to \infty} \Vert (M_i,g_i,p_i)\Vert^{b.h.}_{H^2,r} &\leq \Vert (M,g,p) \Vert^{b.h.}_{H^2,r}, \label{firstdirection1} \\
\liminf\limits_{i\to \infty} \Vert (M_i,g_i,p_i)\Vert^{b.h.}_{H^2,r} &\geq \Vert (M,g,p) \Vert^{b.h.}_{H^2,r}. \label{seconddirection1}
\end{align}
In the following, we discuss \eqref{firstdirection1}, the estimate \eqref{seconddirection1} is proved by an analogous argument left to the reader.

We recall from Definition \ref{def:pointedconv1} that the assumption
$$(M_i,g_i,p_i) \to (M,g,p) \text{ as } i \to \infty \text{ in the pointed $H^2$-topology},$$ 
means that for each integer $i\geq1$ and real $R>0$ a bounded open set $\Om \subset M$ with smooth boundary such that $B_g(p,R) \subset \Om$ together with an embedding
$F_i: \Om \to M_i$
such that $F_i(p) = p_i$, $B_{g_i}(p_i,R) \subset F_i(\Om)$ and
\begin{align} \label{eq:h2conv1}
(F_i)^\ast g_i \to g \text{ as } i \to \infty \text{ in the $H^2$-topology on } \Om.
\end{align}
Using these $F_i$, locally define coordinates on $(M_i,g_i)$. Then approximate these coordinates on $(M_i,g_i)$ by boundary harmonic coordinates with appropriate Dirichlet and Neumann boundary conditions on the boundary, see Lemmas \ref{lemmaharmexpr} and \ref{NeumannHarmonic}. By the standard elliptic estimates in Corollaries \ref{prop:ellest3} and \ref{fullldg6858} in the appendix, and by the fact that due to \eqref{eq:h2conv1} we have, schematically,
\begin{align*}
\triangle_{(F_i)^\ast g_i } \to \triangle_g \text{ as } i \to \infty,
\end{align*}
it follows that these boundary harmonic coordinates on $(M_i,g_i)$ admit the right bounds. For more details, we refer to \cite{Petersen} where the case of manifolds without boundary is treated. \end{proof}


By Lemma \ref{sobolevdomainlemma}, we have the following result. The proof is left to the reader.
\begin{lemma} \label{lem:sobolevmanif}
Let $r>0$. Let $(M,g)$ be a Riemannian $3$-manifold with boundary. Then, for $\tilde r < r$, $\Vert (M,g) \Vert_{C^{0,\a},\tilde{r}}$ for $\a \in (0,1/2)$ is bounded in terms of $\Vert (M,g) \Vert^{b.h.}_{H^2,r}$.
\end{lemma}

As corollary of Theorem \ref{thm:fundamental} and Lemma \ref{lem:sobolevmanif}, we have the following pre-compactness result in the $H^2$-topology. The proof is left to the reader.
\begin{corollary} \label{corollary:fundamentalharm}
Let $Q>0$ and $r>0$ be two reals. Let $\HH^{2}(Q,r)$ denote the class of pointed Riemannian $3$-manifolds $(M,g,p)$ with boundary such that $\Vert (M,g) \Vert^{b.h.}_{H^2,r} \leq Q$. Then $\HH^{2}(Q,r)$ is pre-compact in the pointed $C^{0,\a}$-topology for $\a \in (0,1/2)$. \end{corollary}

The next proposition shows that local $H^2$-convergence in charts implies global pointed $H^2$-convergence of manifolds. The proof is based on the proof of Theorem \ref{thm:fundamental} and left to the reader, see Lemma 4.5 in \cite{Petersen} for the case of manifolds without boundary.

\begin{proposition}\label{prop:localtoglobal}
Let $r>0$ be a fixed real number. Let $(M,g,p)$ and $(M_i,g_i,p_i)$, for $i\geq1$, be smooth Riemannian $3$-manifolds with boundary. Suppose we have countable atlases of $M$ and $M_i$ for $i\geq1$ consisting respectively of boundary harmonic $H^3$-charts
\begin{align*}
\varphi_{n}&: B^+(x_{n},r) \to U \subset M, \\
\varphi_{in}&: B^+(x_{n},r) \to U_i \subset M_i,
\end{align*}
such that 
\begin{enumerate}
\item $\varphi_{in} \to \varphi_n$ in the $H^3$-topology as $i\to \infty$,
\item $\varphi_{n}$ and $\varphi_{in}$ satisfy (h1) and (h3) in Definition \ref{def:harmnorm} for a uniform $Q>0$,
\item $g_{in} \to g_n$ in $H^2(B^+(x_n,r))$, where $g_{in} := {\varphi_{in}}^\ast g_i$ and $g_{n} := {\varphi_{n}}^\ast g$.
\end{enumerate}
Then, subsequentially, 
\begin{align*}
(M_i,g_i, p_i) \to (M,g,p) \text{ as } i \to \infty \text{ in the pointed $H^2$-topology.}
\end{align*}
\end{proposition}


\section{Existence of boundary harmonic coordinates} \label{sec:ExistenceL2}

The following is the main result of this section.
\begin{theorem}[Existence of boundary harmonic coordinates] \label{thm:main1}
Let $(M,h)$ be a compact smooth Riemannian $3$-manifold with boundary satisfying
\begin{align} \begin{aligned}
\Vert \RRRic \Vert_{L^2(M)} \leq \La, \,  \Vert \Th \Vert_{ L^4(\pr M)} \leq \Psi, \, r_{vol}(M,1) \geq v
\end{aligned} \label{eq:bounds1}
\end{align}
for constants $\La \geq0$, $\Psi \geq0$ and $v>0$. Then, for every real $Q>0$, there exists an $$r=r(\La,\Psi,v,Q)>0$$ such that 
$$\Vert (M,h) \Vert^{b.h.}_{H^2,r} \leq Q.$$
Here $\Th$ denotes the second fundamental form of $\pr M \subset M$, and $r_{vol}(M,1)$ is the volume radius at scale $1$ of $(M,h)$, see Definition \ref{def:volumeradius}.
\end{theorem}

\begin{remark} 
Part (1) of Theorem \ref{thm:higherintroexharm} follows directly from Theorem \ref{thm:main1} by picking $Q< \varep$, see also point (3) of Proposition \ref{prop:normprops}.
\end{remark}

Theorem \ref{thm:main1} is proved by contradiction and consists of four steps.
\begin{itemize}
\item Section \ref{contraassump}. Existence of a converging subsequence in the pointed $C^{0,\a}$-topology for $\a \in (0,1/2)$.
\item Section \ref{secstrong111}. Improvement to strong convergence in the $H^2$-topology.
\item Section \ref{rigidity}. Rigidity of the limit manifold.
\item Section \ref{sec:contradiction1}. The contradiction.
\end{itemize}

The Sections \ref{secstrong111} and \ref{rigidity} are independent of each other and lead to the contradiction in Section \ref{sec:contradiction1}.


\subsection{Existence of a converging subsequence in the pointed $C^{0,\a}$-topology for $\a \in (0,1/2)$} \label{contraassump} In this section, we use the contradiction assumption to construct a sequence $(M_i,g_i,p_i)$ of smooth pointed Riemannian manifolds which converge in the pointed $C^{0,\a}$-topology for $\a \in (0,1/2)$ to a limit manifold $(M,g,p)$. Subsequently, in Section the strength of convergence is improved to $H^2$ and in Section \ref{rigidity} the limit manifold $(M,g,p)$ is analysed.

Assume for contradiction there exists
\begin{itemize}
\item a sequence of reals $r_i >0$ such that $\lim\limits_{i\to \infty} r_i = 0$,
\item a sequence $(M_i, h_i)$ of compact smooth Riemannian manifolds with boundary which satisfy \eqref{eq:bounds1} and are such that for all $i$,
\begin{align*}
\Vert (M_i, h_i) \Vert^{b.h.}_{H^2,r_i} > Q.
\end{align*}
\end{itemize}

The $(M_i,h_i)$ are smooth and compact, so their norm is finite. Therefore, using the continuity of the norm, see (2) of Proposition \ref{prop:harmnormprops}, we can decrease $r_i >0$ such that for all $i\geq1$
\begin{align*}
\Vert (M_i, h_i) \Vert^{b.h.}_{H^2,r_i} =Q.
\end{align*}

By (1) of Proposition \ref{prop:harmnormprops}, the conformally rescaled metric $g_i := r_i^{-2} h_i$ satisfies
\begin{align*}
\Vert (M_i, g_i) \Vert^{b.h.}_{H^2,1}= \Vert (M_i, h_i) \Vert^{b.h.}_{H^2,r_i}= Q.
\end{align*} 
First, by Definition \ref{def:harmnorm}, this implies that for each $i\geq 1$ there is a point $p_i \in M_i$ such that
\begin{align*}
\Vert (M_i, g_i,p_i) \Vert^{b.h.}_{H^2,1} \geq Q/2.
\end{align*}
Second, by continuity of the norm, see Proposition \ref{prop:harmnormprops} and \eqref{continuityinr3434}, there is a $Q\leq \tilde{Q} < \infty$ such that for all $i\geq1$
\begin{align}
\Vert (M_i, g_i) \Vert^{b.h.}_{H^2,2}= \tilde{Q} > Q.
\label{boundedQexact}
\end{align} 

At this point, we continue at scale $r=2$ because this gives us the leeway to apply interior elliptic estimates in $\pr M$ to deduce estimates at scale $r=1$.

By \eqref{boundedQexact} and Corollary \ref{corollary:fundamentalharm}, there is a pointed Riemannian manifold with boundary $(M,g,p)$ such that subsequentially as $i\to \infty$,
\begin{align*}
(M_i,g_i,p_i) \to (M,g,p) \text{ as } i \to \infty \text{ in the $C^{0,\a}$-topology}
\end{align*}
for any $\a \in (0,1/2)$. This convergence comes together with countable atlases of $M$ and $M_i$, $i\geq1$, consisting of boundary harmonic $H^3$-charts
\begin{align}\begin{aligned} 
\varphi_{n}: B^{+}(x_n,2) &\to U \subset M, \\ 
\varphi_{in}: B^+(x_{n},2) &\to U_{in} \subset M_i,
\end{aligned} \label{eqatlases23435} \end{align}
respectively, where each $\varphi_{in}$ satisfies (h1)-(h3) of Definition \ref{def:harmnorm} with constant $\tilde Q$, see \eqref{boundedQexact}, and for $\a \in (0,1/2)$,
\begin{align*}
\varphi_{in} \to \varphi_{n} \text{ as } i \to \infty \text{ in the $C^{1,\a}$-topology.}
\end{align*}
Denote the coordinate components of $g_i$, $i\geq1$, and $g$ on $B^+(x_n,2)$ by
\begin{align*}
g_{in} := (\varphi_{in})^\ast g_i, g_{n} := (\varphi_{n})^\ast g.
\end{align*}

We now show three important properties of the constructed sequence. 
\begin{enumerate}

\item By the relation $g_i := r_i^{-2} h_i$ and the uniform bounds \eqref{eq:bounds1}, it follows that as $i \to \infty$,
\begin{align} \begin{aligned}
\Vert \RRRic_{in} \Vert_{\LL^2(B^+(x_n,2))} &\to 0,\\
\Vert \Th_{in} \Vert_{\LL^4(\ubp(x_n,2))} &\to 0,
\end{aligned} \label{eq:convquantzero}
\end{align}
where $\Theta_{in}$ denotes the second fundamental forms of $\ubp(x_n,2)$ with respect to $g_{i}$. In particular, it follows that for the limit Riemannian manifold $(M,g)$,
\begin{align} \label{limitqualbounds3434}
\RRRic=0 \text{ on }M \text{ and } \Th=0, K=0 \text{ on }\pr M,
\end{align}
where $K$ denotes the Gauss curvature of $\pr M$. \\

\item By \eqref{boundedQexact}, the compactness of bounded closed sets for the weak topology and Lemma \ref{sobolevdomainlemma}, it holds that for each $n$, (subsequentially) as $i\to \infty$,
\begin{align} \begin{aligned}
g_{in} &\rightharpoonup g_{n} &&\text{ in } \HH^2(B^+(x_n,2)),&& \\
g_{in} &\to g_{n} &&\text{ in } \HH^1(B^+(x_n,2)),&& \\
g_{in} &\to g_{n} &&\text{ in } \CC^0(B^+(x_n,2)), &&\\
\gd_{in} &\to \gd_n &&\text{ in } \HH^{5/4}(\ubp(x_n,2)) \\
\end{aligned} \label{contraconv1} \end{align}
In particular, we have for the limit metric
\begin{align*}
\Vert g_n \Vert_{\HH^2(B^+(x_n,2))} \lesssim \tilde{Q}.
\end{align*}

\item We have the following \emph{volume growth estimate} on $(M,g)$. For all $r>0$, it holds that 
\begin{align} \label{eqvolumegrowth}
\mathrm{vol}_{g} \Big( B_g(p,r) \Big) \geq \frac{4 \pi}{3} v r^3.
\end{align}

\noindent Indeed, since $(M_i,h_i)$ satisfies \eqref{eq:bounds1}, for each integer $i\geq1$ and point $p_i \in M_i$, we have
\begin{align} \label{vol1i}
\mathrm{vol}_{h_i} \Big( B_{h_i}(p_i,r) \Big) \geq \frac{4 \pi}{3}v r^3 \text{ for all } r\leq1.
\end{align}
For the rescaled $(M_i,g_i)$ this implies that for all reals $0\leq r\leq (r_i)^{-1}$ and points $p_i \in M_i$,
\begin{align*}
\mathrm{vol}_{g_i} \Big( B_{g_i}(p_i,r) \Big) =  \frac{1}{r_i^3} \mathrm{vol}_{h_i} \Big( B_{h_i}(p_i,rr_i) \Big)
\geq  \frac{1}{r_i^3} \frac{4 \pi}{3} v (rr_i)^3
= \frac{4 \pi}{3}v r^3,
\end{align*} 
where we used that $rr_i \leq 1$ to apply \eqref{vol1i}. \\

\noindent Using that $\frac{1}{r_i} \to \infty$ as $i\to \infty$, and $(M_i,g_i,p_i) \to (M,g,p)$ as $i\to \infty$ in the pointed $C^{0,\a}$-topology, we have for each fixed $r\geq 0$ (for $i\geq1$ sufficiently large) as $i\to \infty$,
\begin{align*}
\mathrm{vol}_{g_i} \Big( B_g(p_i,r) \Big) \to \mathrm{vol}_{g} \Big( B_g(p,r) \Big),
\end{align*}
see also the proof of Lemma 11.4.9 in \cite{PetersenBook} for the case of manifolds without boundary. This proves \eqref{eqvolumegrowth}.

\end{enumerate}

\subsection{Improvement to strong convergence in the pointed $H^2$-topology} \label{secstrong111} In this section, we prove that
\begin{align} \label{convstrongab3425}
(M_i, g_i, p_i) \to (M,g,p) \text{ in the pointed } H^2\text{-topology} \text{ as } i \to \infty.
\end{align}
To prove \eqref{convstrongab3425}, it suffices by Proposition \ref{prop:localtoglobal} to show that locally, in the charts $\varphi_{in}$ and $\varphi_n$, see \eqref{eqatlases23435},
\begin{align} \label{strongconvergence}
g_{in} \to g_{n} \text{ in } \HH^2(B^+(x_n,1)) \text{ as } i \to \infty.
\end{align}
Indeed, Proposition \ref{prop:localtoglobal} can be applied because the charts $\varphi_{in}$ have uniform bounds, see \eqref{boundedQexact}, and it is shown independently in Section \ref{rigidity} that $(M,g,p)$ is a smooth pointed Riemannian $3$-manifold. 

For the rest of this section, we prove \eqref{strongconvergence} for a fixed $n$ and hence leave away the index $n$ for ease of presentation. 

We prove \eqref{strongconvergence} in three steps.
\begin{itemize}
\item Section \ref{gaussconvsec}. Strong convergence of the Gauss curvature of $\pr M$. 
\item Section \ref{prelimgdAB}. Strong convergence of the coordinate components $\gd_{AB}$ and $g_{AB}$. Here $\gd$ denotes the induced metric on $\pr M$, see also \eqref{eaexpresseion98976}.
\item Section \ref{g3A}. Strong convergence of the components $g^{33}$ and $g^{A3}$.
\end{itemize}

In the following calculations in Sections \ref{gaussconvsec}-\ref{g3A}, we tacitly use that by \eqref{boundedQexact},
\begin{align*} 
(g_i)^{lm}, (g_i)_{lm}, \sqrt{\det g_i} \text{ for } l,m=1,2,3, \, i \geq1
\end{align*}
are uniformly bounded in $B^+(x,2)$. 


\subsubsection{Strong convergence of the Gauss curvature} \label{gaussconvsec}

Let $K$ and $K_i$ be the Gauss curvature of $\pr M \subset (M,g)$ and $\pr M_i \subset (M_i,g_i)$, respectively. In this section we show that 
$$K_i \to K \text{ in } H^{-1/2}(\ubp(x,11/6)) \text{ as } i\to \infty.$$ 
By \eqref{limitqualbounds3434}, $K=0$ on $\pr M$, and therefore it suffices to prove that
\begin{align}
\Vert K_i \Vert_{H^{-1/2}(\ubp(x,11/6))} \to 0 \text{ as } i \to \infty. \label{eq:gaussconv4}
\end{align} 
However, \eqref{eq:gaussconv4} follows from \eqref{eq:convquantzero} with the trace estimate \eqref{Kestimate}, that is,
\begin{align*} 
\Vert K_i \Vert_{H^{-1/2}(\ubp(x,11/6))} \lesssim \Vert \RRRic_i \Vert_{L^2(B^+(x,2))} + \Vert \Th_i \Vert_{{L^4}(\ubp(x,2))}.
\end{align*}


\subsubsection{Strong convergence of $\gd_{AB}$ and $g_{AB}$ components} \label{prelimgdAB}
In this section we show that as $i\to \infty$,
\begin{align}
\Vert \gd_{iAB} - \gd_{AB} \Vert_{H^{3/2}(\ubp(x,10/6) )} &\to 0, \label{eq:firstconv1230} \\
\Vert g_{iAB} - g_{AB} \Vert_{H^2(B^+(x,9/6))} &\to 0. \label{eq:firstconv123} 
\end{align}

First consider \eqref{eq:firstconv1230}. By Lemma \ref{lemmaharmexpr} and \eqref{limitqualbounds3434}, $\gd_i$ and $\gd$ satisfy on $\ubp(x,2)$
\begin{align} \begin{aligned}
\half \Ld_i \gd_{iAB} + \Qdd_{iAB} &= - K_i, \\
\half \Ld \gd_{AB} + \Qdd_{AB} &= 0.
\end{aligned} \label{eq354288} \end{align}
By the interior elliptic estimate of Theorem \ref{fracellestlit} applied to \eqref{eq354288},
\begin{align} \begin{aligned}
&\Vert \gd_{iAB} - \gd_{AB} \Vert_{H^{3/2}(\ubp(x,10/6))} \\
\lesssim&  \Vert \Ld_i ( \gd_{iAB}- \gd_{AB} )\Vert_{H^{-1/2}(\ubp(x,11/6))} + \Vert \gd_{iAB} -\gd_{AB} \Vert_{H^{1/2}(\ubp(x,11/6))} \\
\lesssim& \Vert (\Ld_i - \Ld) \gd_{AB} \Vert_{H^{-1/2}(\ubp(x,11/6))} + \Vert \Qdd_{iAB}-\Qdd_{AB} \Vert_{H^{-1/2}(\ubp(x,11/6))}  \\
&+ \Vert K_i \Vert_{H^{-1/2}(\ubp(x,11/6))}  + \Vert \gd_{iAB} -\gd_{AB} \Vert_{H^{1/2}(\ubp(x,11/6))}.
\end{aligned} \label{gdest40} \end{align}
By \eqref{contraconv1}, \eqref{eq:gaussconv4} and Lemmas \ref{sobolevdomainlemma} and \ref{LemmaQConvergenceEstimate}, as $i\to \infty$,
\begin{align} \begin{aligned}
\Vert (\Ld_i - \Ld) \gd_{AB} \Vert_{H^{-1/2}(\ubp(x,11/6))} &\to 0, \\
\Vert \Qdd_{iAB}-\Qdd_{AB} \Vert_{H^{-1/2}(\ubp(x,11/6))} &\to 0, \\
\Vert K_i \Vert_{H^{-1/2}(\ubp(x,11/6))} &\to 0.
\end{aligned} \label{eqbdryconvreg3492045} \end{align}

Together \eqref{gdest40} and \eqref{eqbdryconvreg3492045} prove that
\begin{align} \label{gdest4}
\Vert \gd_{iAB} - \gd_{AB} \Vert_{H^{3/2}(\ubp(x,10/6))} \to 0 \text{ as } i \to \infty,
\end{align}
which finishes the proof of \eqref{eq:firstconv1230}.

We turn to the proof of \eqref{eq:firstconv123}. By Lemma \ref{lemmaharmexpr} and \eqref{limitqualbounds3434}, it holds that on $B^+(x,2)$,
\begin{align} \begin{aligned}
\half \triangle_i g_{iAB} + Q_{iAB} &= -\RRRic_{iAB}, \\
\half \triangle g_{AB} + Q_{AB} &= 0.
\end{aligned} \label{eqe534352} \end{align}
By the global elliptic estimates of Theorem \ref{cor:p33est24231} with \eqref{eqe534352}, we have
\begin{align*}
&\Vert g_{iAB} - g_{AB} \Vert_{H^2(B^+(x,9/6))} \\
\lesssim & \Vert \triangle_i (g_{iAB}-g_{AB}) \Vert_{L^2(B^+(x,10/6))} + \Vert g_{iAB}-g_{AB} \Vert_{H^{3/2}(\ubp(x,10/6))} \\
\lesssim &  \Vert (\triangle_i - \triangle) g_{AB} \Vert_{L^2(B^+(x,10/6))} + \Vert Q_{iAB}-Q_{AB} \Vert_{L^2(B^+(x,10/6))} \\
&+ \Vert \RRRic_i \Vert_{L^2(B^+(x,10/6))} + \Vert \gd_{iAB} -\gd_{AB} \Vert_{H^{3/2}(\ubp(x,10/6))} \\
&+ \Vert g_{iAB} - g_{AB} \Vert_{H^1(B^+(x,10/6))},
\end{align*}
where we used that $\gd_{iAB}= g_{iAB}$ on $\ubp(x,2)$. By \eqref{gdest4}, \eqref{eq:convquantzero}, \eqref{contraconv1} and using that by Lemmas \ref{sobolevdomainlemma} and \ref{LemmaQConvergenceEstimate}, as $i \to \infty$,
\begin{align*}
\Vert (\triangle_i - \triangle) g_{AB} \Vert_{L^2(B^+(x,10/6))} &\to0, \\
\Vert Q_{iAB}-Q_{AB} \Vert_{L^2(B^+(x,10/6))} &\to 0, \\
\Vert \RRRic_i \Vert_{L^2(B^+(x,10/6))} &\to 0,
\end{align*}
it follows that 
\begin{align*}
\Vert g_{iAB} - g_{AB} \Vert_{H^2(B^+(x,9/6))} \to 0 \text{ as } i \to \infty,
\end{align*}
which finishes the proof of \eqref{eq:firstconv123}.


\subsubsection{Strong convergence of $g^{33}$ and $g^{3A}$ components} \label{g3A}

In this section, we show that as $i\to \infty$,
\begin{align}
\Vert g_i^{33} - g^{33} \Vert_{H^2(B^+(x,8/6))} &\to 0, \label{eq33cov}\\
\Vert g_i^{3A} - g^{3A} \Vert_{H^2(B^+(x,7/6))} &\to 0.\label{eq3Acov}
\end{align}

On the one hand, by Lemma \ref{lemmaharmexpr} and \eqref{limitqualbounds3434}, we have on $B^+(x,2)$ for $j=1,2,3$
\begin{align}  \begin{aligned}
\half \triangle_{g_i} g_{i}^{3j} + Q_i^{3j} &= -\RRRic_{i}^{33}, \\
\half \triangle_g g^{3j} + Q^{3j} &=0.
\end{aligned} \label{eq93920002} \end{align}

On the other hand, by Lemma \ref{NeumannHarmonic}, on $\ubp(x,2)$ for all $i\geq1$,
\begin{align} \begin{aligned}
N_i \left( g_i^{33} \right) &= 2 \tr \Th_i g_i^{33},  \\
N_i \left( g_i^{3A} \right) &=  \tr \Th_i g_i^{3A} - \half \frac{1}{\sqrt{g_i^{33}}} g_i^{Am} \pr_m g_i^{33}. 
\end{aligned}\label{eq939200022} \end{align}

First consider \eqref{eq33cov}. By the elliptic estimates of Theorem \ref{thm:grisvard1low} with \eqref{eq93920002} we have
\begin{align} \begin{aligned}
&\Vert g_i^{33} - g^{33} \Vert_{H^2(B^+(x,8/6))} \\
\lesssim& \Vert \triangle_g (g^{33} - g_i^{33}) \Vert_{L^2(B^+(x,9/6))} + \Vert g^{33} - g_i^{33} \Vert_{H^1(B^+(x,9/6))} + \Vert N(g^{33} - g_i^{33}) \Vert_{H^{1/2}(\ubp(x,9/6))} \\
\lesssim& \Vert ( \triangle_g- \triangle_{g_i}) g^{33}  \Vert_{L^2(B^+(x,9/6))} + \Vert Q_i^{33}-Q^{33} \Vert_{L^2(B^+(x,9/6))} + \Vert \RRRic_i^{33} \Vert_{L^2(B^+(x,9/6))} \\
& + \Vert g^{33} - g_i^{33} \Vert_{H^1(B^+(x,9/6))}+ \Vert N(g^{33} - g_i^{33}) \Vert_{H^{1/2}(\ubp(x,9/6))} \\
\end{aligned} \label{calceq564} \end{align}
The last term on the right-hand side of \eqref{calceq564} is bounded by \eqref{eq939200022} and Lemmas \ref{sobolevdomainlemma} and \ref{lemmaMoserEstimates} as
\begin{align} \begin{aligned}
& \Vert N(g^{33} - g_i^{33}) \Vert_{H^{1/2}(\ubp(x,9/6))} \\
\lesssim& \Vert g^{33} \tr \Th - g_i^{33} \tr \Th_i \Vert_{H^{1/2}(\ubp(x,9/6))} + \Vert (N-N_i)g^{33} \Vert_{H^{1/2}(\ubp(x,9/6))} \\
\lesssim& \Vert \tr \Th_i \Vert_{H^{1/2}(\ubp(x,9/6))} + \Vert g^{3j} - g_i^{3j} \Vert_{H^{5/4}(\ubp(x,9/6))} \Vert \pr g^{33} \Vert_{H^{1/2}(\ubp(x,9/6))} \\
\lesssim& \Vert \tr \Th_i \Vert_{H^{1/2}(\ubp(x,9/6))} + \Vert g^{3j} - g_i^{3j} \Vert_{H^{5/4}(\ubp(x,9/6))} \Vert g^{33} \Vert_{H^{2}(B^+(x,9/6))},
\end{aligned} \label{EQRRRR1} \end{align}
where we used that $\tr \Th=0$ and $N= -(g^{33})^{-1/2} \nab x^3$, see \eqref{unitnormal} and \eqref{limitqualbounds3434}. 

By \eqref{contraconv1} and Lemmas \ref{sobolevdomainlemma}, \ref{lemmaMoserEstimates} and \ref{Qcalcul}, as $i\to \infty$, 
\begin{align} \begin{aligned}
\Vert ( \triangle_g- \triangle_{g_i}) g^{33}  \Vert_{L^2(B^+(x,9/6))} \to&0, \\
\Vert Q_i^{33}-Q^{33} \Vert_{L^2(B^+(x,9/6))} \to&0.
\end{aligned} \label{EQRRRR2} \end{align}
Moreover, using that in general, $\Th_{AB}= \frac{1}{2a} \pr_3 (g_{AB})  - \frac{1}{2a}\left( \Lied_{\be} \gd \right)_{AB}$, it follows from \eqref{eq:firstconv123} that
\begin{align}
\Vert \tr \Th_i - \tr \Th \Vert_{H^{1/2}(x,9/6)} \to 0 \text{ as } i \to \infty. \label{eqstrongtrthestimateexistence}
\end{align}

Therefore, plugging \eqref{EQRRRR1} into \eqref{calceq564} and using \eqref{EQRRRR2} and \eqref{eqstrongtrthestimateexistence}, we have
\begin{align*}
\Vert g^{33} - g_i^{33} \Vert_{H^2(B^+(x,8/6))} \to 0 \text{ as } i \to \infty,
\end{align*}
which finishes the proof of \eqref{eq33cov}.

The proof of \eqref{eq3Acov} is analogous and left to the reader. 

To summarise, in the above we proved that, as $i\to \infty$,
\begin{align*}
\Vert g_{iAB} - g_{AB} \Vert_{H^2(B^+(x,9/6))} + \Vert g_{i}^{3A} - g^{3A}\Vert_{H^2(B^+(x,8/6))} +  \Vert g_{i}^{3A} - g^{3A}\Vert_{H^2(B^+(x,7/6))}&\to0.
\end{align*}
By Lemma \ref{upstairscontrol}, this implies
\begin{align*}
\Vert g_i - g \Vert_{\HH^2(B^+(x,1))} \to 0 \text{ as } i \to \infty,
\end{align*}
which finishes the proof of \eqref{strongconvergence}.

\subsection{Rigidity of the limit manifold} \label{rigidity}

In this section we prove that the limit manifold
\begin{align} \label{rigidityequality}
(M,g) \simeq (\mathbb{H}^+,e).
\end{align}
The proof of \eqref{rigidityequality} has three steps.
\begin{itemize}
\item Section \ref{extregweaksol1}. Proof of smoothness of $g$.
\item Section \ref{smoothrig}. Doubling of $(M,g)$ and construction of a smooth atlas.
\item Section \ref{sec:manifolddo1}. Rigidity of the limit manifold.
\end{itemize}


\subsubsection{Proof of smoothness of $g$} \label{extregweaksol1} 

We recall from Section \ref{contraassump} that for all points $p \in M$ in the limit manifold $(M,g)$, there exists a boundary harmonic $H^3$-chart $\varphi: B^+(x,1) \to U \subset M$ with $\varphi(x)=p$ such that for the constant $0< \tilde{Q} < \infty$,
\begin{itemize}
\item on $B^+(x,1)$, 
\begin{align*}
e^{-2\tilde{Q}} e_{ij} \leq g_{ij} \leq e^{2\tilde{Q}} e_{ij}.
\end{align*}
\item for all multi-indices $I$ with $1 \leq \vert I \vert \leq 2$, 
\begin{align*}
\max\limits_{i,j=1,2,3} \Vert \pr^I g_{ij} \Vert_{L^2(B^+(x,1))} \leq \tilde{Q}.
\end{align*}
\end{itemize}

By Lemma \ref{lemmaharmexpr}, the metric components $g_{ij}$ satisfy
\begin{align*}
\half \triangle_g g_{ij } + Q_{ij} &= 0 \text{ on } B^+(x,1),\\
\half \Ld_\gd \gd_{AB} + \Qdd_{AB} &= 0 \text{ on } \ubp(x,1).
\end{align*}
These equations can be used to bootstrap regularity to $g \in C^\infty(B^+(x,1))$. We postpone this bootstrapping to Section \ref{sec:harmregularitymflds}, see Proposition \ref{thm:main1Higherreg}. It follows in particular that $(M,g)$ is a smooth Ricci-flat (and hence flat) Riemannian $3$-manifold with boundary.

\begin{remark} \label{uniformbounds02454962111} 
The modulus of continuity of $g$ in $C^k$, for $k\geq 0$, is controlled in terms of $\tilde{Q}$, independently of the point $p\in M$.
\end{remark}

\subsubsection{The manifold double of $(M,g)$} \label{smoothrig} In this section, we define the manifold double $M_{double}$ of $(M,g)$ and show that it is a smooth Riemannian manifold without boundary.

The manifold double $M_{double}$ is topologically defined as
\begin{align*}
M_{double} := \{ 0,1\} \times M
\end{align*}
where the boundaries $\{ 0 \} \times \pr M$ and $\{1\} \times \pr M$ are identified.

Given that $(M,g)$ is a Riemannian manifold with boundary, the sets $\{ 0\} \times M$ and $\{ 1\} \times M$, equipped with metric $g$, are also Riemannian manifolds with boundary.

In the following, we construct interior charts of $M_{double}$ centered on points $p\in \{ 0 \} \times \pr M = \{1\} \times \pr M$ and show that the metric $g$ on $M$ extends smoothly to a Riemannian metric on $M_{double}$.


Let $p \in \{ 0\} \times \pr M$. By Section \ref{extregweaksol1}, there is a boundary harmonic chart
\begin{align*}
\varphi: B^+(0,1) \to U \subset M
\end{align*}
with $\varphi(0)=p$ such that $g_{ij} \in C^{\infty}(B^+(0,1))$.

Using this smoothness, we can construct so-called \emph{Gaussian normal coordinates} $(z^1,z^2,z^3)$ in a neighbourhood $\UU \subset \{0\} \times M$ of $p$, see for example Section 3.3 in \cite{Wald}. They are such that for a small real $\de>0$,
\begin{align*}
\mathcal{U} &= \{ z \in \ubp(0,1/2) \times [0,\de ]\}, \\
\mathcal{U} \cap \pr M &= \{ z\in \ubp(0,1/2) \times \{0\}\},\\
\nab_{\pr_{z^3}} \pr_{z^3} &=0, \\
\pr_{z^3} \vert_{\mathcal{U} \cap \pr M} &\text{ is normal to } \mathcal{U} \cap \pr M, \\
g(\pr_{z^3}, \pr_{z^3})&=1.
\end{align*}
We note that the scale $\de>0$ can be bounded from below by the modulus of continuity of the metric $g$, that is, $\tilde{Q}$, see also Remark \ref{uniformbounds02454962111}. In these coordinates, the metric $g$ is given by 
\begin{align} \label{pullbackin24343}
g =  (dz^3)^2 + \gd^z_{AB} dz^A dz^B,
\end{align}
where $\gd^z$ denotes the induced metric on level sets of $z^3$. These coordinates are also called \emph{zero-shift coordinates}, because $\be=0$ in $\ubp(0,1/2) \times [0, \de]$, see \eqref{eaexpresseion98976}.

These coordinates are a local boundary chart of $ \{ 0\} \times M$ around $p$. By applying the same construction around $p$ in $ \{ 1\} \times M$ and identifying the coordinates $z^{1},z^2$ on $\pr M$, we can extend the above coordinates to coordinates on $M_{double}$, with chart
\begin{align} \label{eqcoordinatechart0340235}
\varphi: \ubp(0,1/2) \times [-\de,\de] \to M_{double}.
\end{align}
Indeed, the geodesic construction of $z^3$ and the consequent transport of $(z^1,z^2)$ along $z^3$ yields smooth coordinates around $p$ in $M_{double}$. \\

The pullback of $g$ by the chart $\varphi: \ubp(0,1/2) \times [-\de,\de] \to M_{double}$ in \eqref{eqcoordinatechart0340235} is given by 
\begin{align} \begin{aligned}
(\varphi)^*g_{ij}(z^1,z^2,z^3) := \begin{cases} g_{ij}(z^1,z^2,z^3), & \text{if } z^3\geq 0, \\ 
g_{ij}(z^1,z^2,-z^3) &\text{else}, \end{cases}
\end{aligned} \label{eq:defexteh}\end{align}
where $g_{ij}$ on $\ubp(0,1/2) \times [0,\de]$ are the pullback components of $g$ as in \eqref{pullbackin24343}.

\begin{claim} \label{claimsmooth5482}
The pullback metric $g$ in \eqref{eq:defexteh} is smooth on $\ubp(0,1/2) \times [-\de,\de]$.
\end{claim}
It suffices to prove that on $\ubp(0,1/2) \times \{0\}$
\begin{align} \label{dervan1}
\pr^m_{z^3} g_{AB} = 0
\end{align}
for all $m \geq 0$.

On the one hand, on all of $\ubp(0,1/2) \times [-\de, \de]$, it holds that
\begin{align} \label{drel1}
g(\pr_{z^A}, \nab_{\pr_{z^3}} \pr_{z^B}) = \half \left( \pr_{z^A} g_{3B} + \pr_{z^3}g_{AB} - \pr_{z^B} g_{3A} \right) =\half \pr_{z^3} g_{AB},
\end{align}
where we used that $g_{3A}=0$ in the zero-shift coordinates. Moreover, differentiating \eqref{drel1} by $\pr_{z^3}$ yields
\begin{align*}
\half \pr^2_{z^3} g_{AB} =& g(\nab_{\pr_{z^3}} \pr_{z^A}, \nab_{\pr_{z^3}} \pr_{z^B}) + g(\pr_{z^A}, \nab_{\pr_{z^3}}^2 \pr_{z^B}) \\
=&g(\nab_{\pr_{z^3}} \pr_{z^A}, \nab_{\pr_{z^3}} \pr_{z^B})+ g(\pr_{z^A},\nab_{\pr_{z^3}} \nab_{\pr_{z^B}} \pr_{z^3}) \\
=& g(\nab_{\pr_{z^3}} \pr_{z^A}, \nab_{\pr_{z^3}} \pr_{z^B}) + g(\pr_{z^A},\mathrm{Rm}(\pr_{z^3}, \pr_{z^B}) \pr_{z^3}) + g(\pr_{z^A},\nab_{\pr_{z^B}}\nab_{\pr_{z^3}}\pr_{z^3}) \\
=&g(\nab_{\pr_{z^3}} \pr_{z^A}, \nab_{\pr_{z^3}} \pr_{z^B}),
\end{align*}
where we used that $(M,g)$ is flat and $\pr_{z^3}$ geodesic.

On the other hand, because $\be=0$ and $\Th=0$ on $z^3=0$, it follows by \eqref{secondformexpr3} that on $\{z^3=0\}$
$$\pr_{z^3} g_{AB} =0, \qquad \nab_{\pr_{z^3}} \pr_{z^B}=0.$$
This shows that $\pr_{z^3}^2 g_{AB}=0$ at $\{z^3=0\}$. The statement \eqref{dervan1} follows by continued differentiation of \eqref{drel1}. This finishes the proof of Claim \ref{claimsmooth5482}.
This finishes the proof that $(M_{double},g)$ is a smooth flat Riemannian $3$-manifold without boundary. Further, by the fact that $(M,g)$ is a complete Riemannian manifold, it follows that $(M_{double},g)$ is also complete. 


\subsubsection{Conclusion of the rigidity of the limit manifold} \label{sec:manifolddo1} In this section we conclude that 
\begin{align} \label{id1id}
(M_{double},g) \simeq (\RRR^3,e).
\end{align}

We recall that in the previous sections, we showed that $(M_{double},g)$ is a complete smooth flat Riemannian $3$-manifold without boundary. We further recall that by \eqref{eqvolumegrowth} there is a $v>0$ such that for all $r>0$ and $p\in M$,
\begin{align*} 
\mathrm{vol}_{g} \Big( B_g(p,r) \Big) \geq \frac{4\pi}{3} v r^3.
\end{align*}
This lower bound on the volume growth continues to hold on the doubled $(M_{double}, g)$.

\begin{proposition} 
Let $(M,g)$ be a complete smooth flat Riemannian $3$-manifold such that there is a $v>0$ so that for all reals $r>0$ and points $p \in M$,
\begin{align*}
\mathrm{vol}_{g} \Big( B_g(p,r) \Big) \geq v r^3.
\end{align*}
Then, 
 \begin{align*}
(M,g) \simeq (\RRR^3,e).
\end{align*}
\end{proposition}
A proof of this proposition can be found at the end of the proof of Lemma 11.4.9 in \cite{PetersenBook}, see also Theorem 5.4 in \cite{Petersen}.
 
By \eqref{id1id}, the hypersurface $\pr M \subset (\RRR^3,e)$, and further $\Th=0$ on $\pr M$. Therefore $\pr M$ equals a plane in $(\RRR^3,e)$. This shows that $(M,g) \simeq (\mathbb{H}^+,e)$ and in particular,
\begin{align*}
\Vert (M,g) \Vert^{b.h.}_{H^2,1} =0.
\end{align*}


\subsection{The contradiction} \label{sec:contradiction1} In this section, we conclude the proof of Theorem \ref{thm:main1}.

On the one hand, by Sections \ref{contraassump} and  \ref{secstrong111}, the sequence $(M_i,g_i,p_i)$ is such that
\begin{align*}
\Vert (M_i, g_i,p_i) \Vert^{b.h.}_{H^2,1} \geq Q/2,
\end{align*}
and
\begin{align*}
(M_i,g_i,p_i) \to (M,g,p) \text{ as } i \to \infty \text{ in the pointed $H^2$-topology}.
\end{align*}
By Proposition \ref{prop:harmnormprops}, this implies that 
\begin{align*}
\Vert (M,g,p) \Vert^{b.h.}_{H^2,1} \geq Q/2 \neq 0.
\end{align*}

On the other hand, by Section \ref{rigidity}, $(M,g) \simeq (\HHH^+,e)$, so that
\begin{align*}
\Vert (M,g) \Vert^{b.h.}_{H^2,1} = 0.
\end{align*}
This contradiction finishes the proof of Theorem \ref{thm:main1}.


\section{Higher regularity estimates for boundary harmonic coordinates} \label{sec:harmregularitymflds}

The main result of this section is the following. 
\begin{proposition}[Higher regularity of boundary harmonic coordinates] \label{thm:main1Higherreg}
Let $g$ be a Riemannian metric in boundary harmonic coordinates on $B^+(x,r)$ such that for some $\varep_0>0$,
\begin{align*}
\Vert g-e\Vert_{\HH^2(B^+(x,r))} < \varep_0.
\end{align*}
Assume in addition that for an integer $m\geq0$,
\begin{align*}
\sum\limits_{i=0}^m \Vert \nab^{(i)} \RRRic \Vert_{\LL^2(B^+(x,r))} < \infty.
\end{align*}
Then, there exists $\varep>0$ such that if $\varep_0< \varep$, then for all reals $0<r'<r$,
\begin{align*} 
\Vert g \Vert_{\HH^{m+2}(B^+(x,r'))} \leq C_{r',r} \sum\limits_{i=0}^m \Vert \nab^{(i)} \RRRic \Vert_{\LL^2(B^+(x,r))} + C_{r',r,m} \Vert g \Vert_{\HH^2(B^+(x,r))}.
\end{align*}
\end{proposition}


Before proving Proposition \ref{thm:main1Higherreg}, we first prove the higher regularity estimates of Theorem \ref{thm:higherintroexharm}.
\begin{proof}[Proof of part (2) of Theorem \ref{thm:higherintroexharm}]
First, by Theorem \ref{thm:main1}, there is a real 
$$r_0=r_0(\Vert \RRRic \Vert_{L^2(M)},\Vert \Th \Vert_{ L^4(\pr M)},r_{vol}(M,1),  \varep/2)>0$$
such that for each $p\in M$ there exists a boundary harmonic chart $\varphi: B^+(x,r_0) \to U \subset M$ with $\varphi(x)=p$ such that on $B^+(x,r_0)$,
\begin{align*} 
(1-\varep/2)e_{ij} \leq g_{ij} \leq (1+ \varep/2) e_{ij}
\end{align*}
and
\begin{align*} \begin{aligned}
r_0^{-1/2} \Vert \pr g \Vert_{\LL^2(B^+(x,r_0))} + r_0^{1/2} \Vert \pr^2 g \Vert_{\LL^2(B^+(x,r_0))} \Big) \leq \varep/2.
 \end{aligned} 
  \end{align*}
Therefore, by restriction, it holds in particular on $B^+(x,r_0/2)$
 \begin{align} \label{eq158283}
(1-\varep)e_{ij} \leq g_{ij} \leq (1+ \varep) e_{ij},
\end{align}
and
\begin{align} \begin{aligned}
\left( \frac{r_0}{2}\right)^{-1/2} \Vert \pr g \Vert_{\LL^2(B^+(x,r_0/2))} +\left( \frac{r_0}{2} \right)^{1/2} \Vert \pr^2 g \Vert_{\LL^2(B^+(x,r_0/2))} &\leq \varep.
\end{aligned} \label{eqhigherreggoal2222} \end{align}

Second, by Proposition \ref{thm:main1Higherreg}, 
\begin{align}  \begin{aligned}
\Vert g \Vert_{\HH^{m+2}(B^+(x,r_0/2))} \leq& C_{r_0} \sum\limits_{i=0}^m \Vert \nab^{(i)} \RRRic \Vert_{\LL^2(B^+(x,r_0))} +C_{r_0,m} \Vert g \Vert_{\HH^2(B^+(x,r_0))} \\ 
\leq& C_{r_0} \sum\limits_{i=0}^m \Vert \nab^{(i)} \RRRic \Vert_{\LL^2(M)} + C_{r_0,m} \varep,
\end{aligned} \label{higherregest122222} 
\end{align}

Setting $r:=\frac{r_0}{2}$, it follows that \eqref{eq158283}, \eqref{eqhigherreggoal2222} and \eqref{higherregest122222} prove part (2) of Theorem \ref{thm:higherintroexharm}.
\end{proof}


We turn now to the proof of Proposition \ref{thm:main1Higherreg}. The higher regularity estimates are derived by applying standard elliptic estimates to the equations (see Lemma \ref{lemmaharmexpr})
\begin{align}
&&\half \Ld_\gd \gd_{AB} + \Qdd_{AB}  &= -2K &\text{ on } \ubp(x,r),&& \label{ellgd} \\
&&\half \triangle g_{AB} + Q_{AB} &= -\RRRic_{AB} &\text{ on } B^+(x,r), && \label{ellg}\\
&&\half \triangle g^{3j} + Q^{3j} &= \RRRic^{3j} &\text{ on } B^+(x,r).&& \label{ellg3}
\end{align}

For completeness, we outline the proof of Proposition \ref{thm:main1Higherreg} in the rest of this section. The proof of Proposition \ref{thm:main1Higherreg} is based on induction.
\begin{itemize}
\item \emph{The induction basis.} The case $m=0$ is direct.
\item \emph{The induction step.} Let $m\geq1$. It suffices to prove that if $\varep>0$ is sufficiently small, then for $0<r'<r$, 
\begin{align} \begin{aligned}
\Vert \pr^{m+2} g \Vert_{L^2(B^+(x,r'))} \leq& C_{r',r} \Big( \Vert \nab^{(m)} \RRRic \Vert_{L^2(B^+(x,r))} +\Vert g \Vert_{\HH^{m+1}(B^+(x,r))} \Big)\\
& + C_{r',r,m} \Vert g-e \Vert_{\HH^2(B^+(x,r))}.
\end{aligned} \label{ineqtoprove4343} \end{align}
\end{itemize} 

In the following, we prove \eqref{ineqtoprove4343} by using standard elliptic estimates. Let $\chi := \chi_{x,r',r}$ be the smooth cut-off function defined in \eqref{defchi}. For ease of notation, the constant in $\lesssim$ is here allowed to depend on $r'$ and $r$, and we tactily use Lemma \ref{sobolevtracedomain} and standard product estimates as in Lemma \ref{lemmaMoserEstimates}. We assume throughout the proof that $\varep>0$ is sufficiently small. \\


\noindent {\bf Control of $\gd_{AB}$.} By Theorem \ref{fullldg6858} applied to \eqref{ellgd}, we have 
\begin{align} \begin{aligned}
&\Vert \prd^{m-1} \gd_{AB} \Vert_{H^{5/2}(\ubp(x,r'))} \\
 \leq& \Vert \prd^{m-1} (\chi \gd_{AB}) \Vert_{H^{5/2}(\ubp(x,r))} \\
\lesssim& \Vert \Ld_\gd (\prd^{m-1} (\chi \gd_{AB}) ) \Vert_{H^{1/2}(\ubp(x,r))} \\
\lesssim& \Vert \chi \prd^{m-1} (\Ld_\gd \gd_{AB}) \Vert_{H^{1/2}(\ubp(x,r))} + \Vert \gd-e \Vert_{\HH^{3/2}(\ubp(x,r))} \Vert \prd^{m-1}( \chi \gd) \Vert_{\HH^{3/2}(\ubp)} \\
&+ \Vert \gd \Vert_{\HH^{m+1/2}(\ubp(x,r))} + C_m \Vert \gd-e \Vert_{\HH^{3/2}(\ubp(x,r))},
\end{aligned} \label{higherreg232388first} \end{align}
where we estimated
\begin{align*}
\Vert [\Ld_\gd, \prd^{m-1}] (\chi \gd_{AB}) \Vert_{H^{1/2}(\ubp(x,r))} =&\Vert [\gd^{CD} \prd_C \prd_D, \prd^{m-1}] (\chi \gd_{AB}) \Vert_{H^{1/2}(\ubp(x,r))} \\
\lesssim& \Vert \gd-e \Vert_{\HH^{3/2}(\ubp(x,r))} \Vert \prd^{m-1}(\chi \gd) \Vert_{\HH^{3/2}(\ubp(x,r))} \\
&+ \Vert \gd \Vert_{\HH^{m+1/2}(\ubp(x,r))} + C_m \Vert \gd-e \Vert_{\HH^{3/2}(\ubp(x,r))}.
\end{align*}

\ni The first term on the right-hand side of \eqref{higherreg232388first} is estimated by using \eqref{ellgd} and the twice traced Gauss equation $2K = \mathrm{R}_{scal} - 2 \RRRic(N,N) + (\tr \Th)^2 - \vert \Th \vert^2$,
\begin{align} \begin{aligned}
&\Vert \chi \prd^{m-1} ( \Ld_\gd \gd_{AB} ) \Vert_{H^{1/2}(\ubp(x,r))} \\
\leq& \Vert \chi \prd^{m-1} \Qdd_{AB} \Vert_{H^{1/2}(\ubp(x,r))} + \Vert \chi \prd^{m-1} (\Rscal - 2 \RRRic_{NN}) \Vert_{H^{1/2}(\ubp(x,r))} \\
&+ \Vert \chi \prd^{m-1} \left((\tr \Th)^2 - \vert \Th \vert^2 \right) \Vert_{H^{1/2}(\ubp(x,r))} \\
\lesssim& \Vert \nab^{(m)} \RRRic \Vert_{\LL^2(B^+(x,r))} +  \Vert g-e \Vert_{\HH^2(B^+(x,r))} \Vert \pr^{m+2} (\chi g) \Vert_{\LL^2(B^+(x,r))} \\
&+ \Vert g \Vert_{\HH^{m+1}(B^+(x,r))} + C_{m} \Vert g-e \Vert_{\HH^2(B^+(x,r))}.
\end{aligned} \label{EQrrr1} \end{align}

\ni Plugging \eqref{EQrrr1} into \eqref{higherreg232388first} yields
\begin{align} \begin{aligned}
& \Vert \prd^{m-1} (\chi \gd_{AB}) \Vert_{\HH^{5/2}(\ubp(x,r))} \\
\lesssim& \Vert \nab^{(m)} \RRRic \Vert_{\LL^2(B^+(x,r))} +\Vert g-e \Vert_{\HH^2(B^+(x,r))} \Vert \pr^{m+2}(\chi g) \Vert_{\LL^2(B^+(x,r))} \\
&+ \Vert g \Vert_{\HH^{m+1}(B^+(x,r))} + C_m \Vert g-e \Vert_{\HH^2(B^+(x,r))}.
\end{aligned} \label{gdhigherreg2483534} \end{align}


\ni {\bf Control of $g_{AB}$.} By a standard application of Theorem \ref{cor:p33est24231} applied to \eqref{ellg}, we have
\begin{align} \begin{aligned}
&\Vert \pr^{m} g_{AB} \Vert_{H^2(B^+(x,r'))} \\
\leq &\Vert \pr^{m}(\chi g_{AB}) \Vert_{H^2(B^+(x,r))} \\
\lesssim& \Vert \chi \pr^m(\triangle_g g_{AB}) \Vert_{L^2(B^+(x,r))} + \Vert \prd^{m}(\chi g_{AB}) \Vert_{H^{3/2}(\ubp(x,r))} \\
&+ \Vert g-e \Vert_{\HH^2(B^+(x,r))} \Vert \pr^{m+2}(\chi g) \Vert_{\LL^2(B^+(x,r))} \\
&+ \Vert g \Vert_{\HH^{m+1}(B^+(x,r))} + C_m \Vert g-e \Vert_{\HH^2(B^+(x,r))},
\end{aligned} \label{EQrrr3} \end{align}
where we estimated
\begin{align*}
\Vert [\triangle_g, \pr^{m}] (\chi g_{AB}) \Vert_{H^{3/2}(\ubp(x,r))} \lesssim& \Vert g-e \Vert_{\HH^2(B^+(x,r))} \Vert \pr^{m+2}(\chi g) \Vert_{\LL^2(B^+(x,r))} \\
&+ \Vert g \Vert_{\HH^{m+1}(B^+(x,r))} + C_m \Vert g-e \Vert_{\HH^2(B^+(x,r))}.
\end{align*}

\ni The first term on the right-hand side of \eqref{EQrrr3} is estimated as above by using \eqref{ellg},
\begin{align} \begin{aligned}
&\Vert \chi \pr^m(\triangle_g g_{AB}) \Vert_{L^2(B^+(x,r))} \\
\lesssim& \Vert \chi \pr^m \RRRic \Vert_{\LL^2(B^+(x,r))} + \Vert \chi \pr^m Q_{AB} \Vert_{L^2(B^+(x,r))} \\
\lesssim& \Vert \nab^{(m)} \RRRic \Vert_{\LL^2(B^+(x,r))} + \Vert g-e \Vert_{\HH^2(B^+(x,r))} \Vert \pr^{m+2} (\chi g ) \Vert_{L^2(B^+(x,r))} \\
&+ \Vert g \Vert_{\HH^{m+1}(B^+(x,r))} + C_m \Vert g-e \Vert_{\HH^2(B^+(x,r))}.
\end{aligned} \label{EQrrr4} \end{align}

\ni The second term on the right-hand side of \eqref{EQrrr3} is estimated by using that $g_{AB} = \gd_{AB}$ on $\ubp(x,r)$ and \eqref{gdhigherreg2483534}, yielding
\begin{align} \begin{aligned}
&\Vert \prd^{m}(\chi g_{AB}) \Vert_{H^{3/2}(\ubp(x,r))} \\
=& \Vert \prd^{m}(\chi \gd_{AB}) \Vert_{H^{3/2}(\ubp(x,r))} \\
\lesssim& \Vert \nab^{(m)} \RRRic \Vert_{\LL^2(B^+(x,r))} +\Vert g-e \Vert_{\HH^2(B^+(x,r))} \Vert \pr^{m+2}(\chi g) \Vert_{\LL^2(B^+(x,r))} \\
&+ \Vert g \Vert_{\HH^{m+1}(B^+(x,r))} + C_m \Vert g-e \Vert_{\HH^2(B^+(x,r))}.
\end{aligned} \label{EQrrr5} \end{align}

\ni Plugging \eqref{EQrrr4} and \eqref{EQrrr5} into \eqref{EQrrr3} yields
\begin{align} \begin{aligned}
&\Vert \pr^{m}(\chi g_{AB}) \Vert_{H^2(B^+(x,r))} \\
 \lesssim& \Vert \nab^{(m)} \RRRic \Vert_{\LL^2(B^+(x,r))} +\Vert g-e \Vert_{\HH^2(B^+(x,r))} \Vert \pr^{m+2}(\chi g) \Vert_{\LL^2(B^+(x,r))} \\
&+ \Vert g \Vert_{\HH^{m+1}(B^+(x,r))} + C_m \Vert g-e \Vert_{\HH^2(B^+(x,r))}.
\end{aligned} \label{frist8349215version1} \end{align}

\ni In particular, \eqref{frist8349215version1} implies by $\Theta_{AB} = \frac{1}{2a} \pr_3 (g_{AB}) - \frac{1}{2a} (\Lied_\be \gd)_{AB}$, see \eqref{secondformexpr3}, that
\begin{align} \begin{aligned}
&\Vert \chi \prd^{m} \tr \Theta \Vert_{H^{1/2}(\ubp(x,r))} \\
\lesssim& \Vert \nab^{(m)} \RRRic \Vert_{\LL^2(B^+(x,r))} +\Vert g-e \Vert_{\HH^2(B^+(x,r))} \Vert \pr^{m+2}(\chi g) \Vert_{\LL^2(B^+(x,r))} \\
&+ \Vert g \Vert_{\HH^{m+1}(B^+(x,r))} + C_m \Vert g-e \Vert_{\HH^2(B^+(x,r))}.
\end{aligned} \label{trthest2353} \end{align}


\ni {\bf Control of $g^{33}$.} By a standard application of Theorem \ref{thm:grisvard1low} to \eqref{ellg3}, we have
\begin{align} \begin{aligned}
&\Vert \pr^{m} g^{33} \Vert_{H^2(B^+(x,r'))} \\
\leq &\Vert \pr^{m}( \chi g^{33}) \Vert_{H^2(B^+(x,r))} \\
\lesssim& \Vert \chi \pr^m (\triangle_g g^{33}) \Vert_{L^2(B^+(x,r))} + \Vert \chi \prd^{m} N(g^{33}) \Vert_{H^{1/2}(\ubp(x,r))} \\
&+ \Vert g-e \Vert_{\HH^2(B^+(x,r))} \Vert \pr^{m+2}(\chi g) \Vert_{\LL^2(B^+(x,r))} \\
&+ \Vert g \Vert_{\HH^{m+1}(B^+(x,r))} + C_m \Vert g-e \Vert_{\HH^2(B^+(x,r))}.
\end{aligned}\label{estimateg332323} \end{align}

\ni The first term on the right-hand side of \eqref{estimateg332323} is estimated by using \eqref{ellg3},
\begin{align} \begin{aligned}
&\Vert \chi \pr^m (\triangle_g g^{33}) \Vert_{L^2(B^+(x,r))} \\
\lesssim& \Vert \nab^{(m)} \RRRic \Vert_{\LL^2(B^+(x,r))} +\Vert g-e \Vert_{\HH^2(B^+(x,r))} \Vert \pr^{m+2}(\chi g) \Vert_{\LL^2(B^+(x,r))} \\
&+ \Vert g \Vert_{\HH^{m+1}(B^+(x,r))} + C_m \Vert g-e \Vert_{\HH^2(B^+(x,r))}.
\end{aligned} \label{EQrrr6} \end{align}

\ni The second term on the right-hand side of \eqref{estimateg332323} is estimated by using that $$N(g^{33}) = 2 g^{33} \tr \Theta$$ in boundary harmonic coordinates, see Lemma \ref{NeumannHarmonic}, together with the previous estimate \eqref{trthest2353},
\begin{align} \begin{aligned}
&\Vert \chi \prd^{m} N(g^{33}) \Vert_{H^{1/2}(\ubp(x,r))} \\
=& \Vert \chi \prd^{m} (2g^{33} \tr \Theta) \Vert_{H^{1/2}(\ubp(x,r))} \\
 \lesssim&  \Vert \nab^{(m)} \RRRic \Vert_{\LL^2(B^+(x,r))} +\Vert g-e \Vert_{\HH^2(B^+(x,r))} \Vert \pr^{m+2}(\chi g) \Vert_{\LL^2(B^+(x,r))} \\
&+ \Vert g \Vert_{\HH^{m+1}(B^+(x,r))} + C_m \Vert g-e \Vert_{\HH^2(B^+(x,r))}.
\end{aligned} \label{EQrrr7} \end{align}

\ni Plugging \eqref{EQrrr6} and \eqref{EQrrr7} into \eqref{estimateg332323} yields
\begin{align} \begin{aligned} 
\Vert \pr^{m-1}(\chi g^{33}) \Vert_{W^{2,3}(B^+(x,r))} \lesssim& \Vert g-e \Vert_{\HH^2(B^+(x,r))} \Vert \pr^{m+2}(\chi g) \Vert_{\LL^2(B^+(x,r))} \\
&+ \Vert \chi \pr^{m} \RRRic \Vert_{\LL^2(B^+(x,r))} + \Vert g \Vert_{\HH^{m+1}(B^+(x,r))} \\
&+ C_m \Vert g-e \Vert_{\HH^2(B^+(x,r))}.
\end{aligned} \label{g33est35423} \end{align}

\ni {\bf Control of $g^{3A}$.} By a standard application of Theorem \ref{thm:grisvard1low} to \eqref{ellg3}, we have
\begin{align} \begin{aligned}
&\Vert \pr^{m} g^{3A} \Vert_{H^2(B^+(x,r'))} \\
\leq &\Vert \pr^{m}( \chi g^{3A}) \Vert_{H^2(B^+(x,r))} \\
\lesssim& \Vert \chi \pr^m (\triangle_g g^{3A}) \Vert_{L^2(B^+(x,r))} + \Vert \chi \prd^{m} N(g^{3A}) \Vert_{H^{1/2}(\ubp(x,r))} \\
&+ \Vert g-e \Vert_{\HH^2(B^+(x,r))} \Vert \pr^{m+2}(\chi g) \Vert_{\LL^2(B^+(x,r))} \\
&+ \Vert g \Vert_{\HH^{m+1}(B^+(x,r))} + C_m \Vert g-e \Vert_{\HH^2(B^+(x,r))}.
\end{aligned}\label{estimateg332323111} \end{align}

\ni The first term on the right-hand side of \eqref{estimateg332323111} is estimated by using \eqref{ellg3},
\begin{align} \begin{aligned}
&\Vert \chi \pr^m (\triangle_g g^{3A}) \Vert_{L^2(B^+(x,r))} \\
\lesssim& \Vert \nab^{(m)} \RRRic \Vert_{\LL^2(B^+(x,r))} +\Vert g-e \Vert_{\HH^2(B^+(x,r))} \Vert \pr^{m+2}(\chi g) \Vert_{\LL^2(B^+(x,r))} \\
&+ \Vert g \Vert_{\HH^{m+1}(B^+(x,r))} + C_m \Vert g-e \Vert_{\HH^2(B^+(x,r))}.
\end{aligned} \label{EQrrr6111} \end{align}

\ni The second term on the right-hand side of \eqref{estimateg332323111} is estimated by using that 
\begin{align*}
N(g^{3A}) = g^{3A} \tr \Theta - \half \frac{1}{\sqrt{g^{33}}} g^{3i} \pr_i g^{33},
\end{align*}
in boundary harmonic coordinates, see Lemma \ref{NeumannHarmonic}, together with \eqref{trthest2353} and \eqref{g33est35423},
\begin{align} \begin{aligned}
&\Vert \chi \prd^{m} N(g^{33}) \Vert_{H^{1/2}(\ubp(x,r))} \\
=& \Vert \chi \prd^{m} (2g^{33} \tr \Theta) \Vert_{H^{1/2}(\ubp(x,r))} \\
 \lesssim&  \Vert \nab^{(m)} \RRRic \Vert_{\LL^2(B^+(x,r))} +\Vert g-e \Vert_{\HH^2(B^+(x,r))} \Vert \pr^{m+2}(\chi g) \Vert_{\LL^2(B^+(x,r))} \\
&+ \Vert g \Vert_{\HH^{m+1}(B^+(x,r))} + C_m \Vert g-e \Vert_{\HH^2(B^+(x,r))}.
\end{aligned} \label{EQrrr7111} \end{align}

\ni Plugging \eqref{EQrrr6111} and \eqref{EQrrr7111} into \eqref{estimateg332323111} yields
\begin{align} \begin{aligned} 
&\Vert \pr^{m}(\chi g^{3A}) \Vert_{H^2(B^+(x,r))} \\
 \lesssim&  \Vert \nab^{(m)} \RRRic \Vert_{\LL^2(B^+(x,r))} +\Vert g-e \Vert_{\HH^2(B^+(x,r))} \Vert \pr^{m+2}(\chi g) \Vert_{\LL^2(B^+(x,r))} \\
&+ \Vert g \Vert_{\HH^{m+1}(B^+(x,r))} + C_m \Vert g-e \Vert_{\HH^2(B^+(x,r))}.
\end{aligned} \label{g33est35423111} \end{align}

\ni {\bf Conclusion of the proof of \eqref{ineqtoprove4343}.} By Lemma \ref{upstairscontrol}, the estimates \eqref{gdhigherreg2483534}, \eqref{frist8349215version1}, \eqref{g33est35423} and \eqref{g33est35423111} imply that
\begin{align*}
&\Vert \pr^{m}( \chi g) \Vert_{\HH^2(B^+(x,r))} \\
\lesssim& \Vert \nab^{(m)} \RRRic \Vert_{\LL^2(B^+(x,r))} +\Vert g-e \Vert_{\HH^2(B^+(x,r))} \Vert \pr^{m+2}(\chi g) \Vert_{\LL^2(B^+(x,r))} \\
&+ \Vert g \Vert_{\HH^{m+1}(B^+(x,r))} + C_m \Vert g-e \Vert_{\HH^2(B^+(x,r))} \\
\lesssim& \Vert \nab^{(m)} \RRRic \Vert_{\LL^2(B^+(x,r))} + \Vert g \Vert_{\HH^{m+1}(B^+(x,r))} + C_m \Vert g-e \Vert_{\HH^2(B^+(x,r))} ,
 \end{align*}
where we used that $\varep>0$ is sufficiently small to absorb the second term in the left-hand side. This finishes the proof of \eqref{ineqtoprove4343} and concludes the induction step. This finishes the proof of Proposition \ref{thm:main1Higherreg}.

\appendix

\section{Global elliptic estimates} In this section, we collect global elliptic estimates for Dirichlet and Neumann problems, see Sections \ref{secDirEll} and \ref{secNeuEll}, respectively. These estimates are applied in Section \ref{sec:setup} to construct boundary harmonic coordinates, and in Section \ref{sec:harmregularitymflds} to derive higher regularity estimates for given boundary harmonic coordinates. \\

The estimates are standard, see for example \cite{GilbargTrudinger} or \cite{Petersen}, but we give some proofs for completeness. In the following, the constant in $\lesssim$ is allowed to depend on the domain $\Om$.


\subsection{Global estimates for Dirichlet data} \label{secDirEll}

We have the following standard elliptic estimates, see for example Theorem 9.13 from \cite{GilbargTrudinger} and Lemma \ref{sobolevdomainlemma}.
\begin{theorem}[Standard global elliptic estimate] \label{cor:p33est24231}
Let $n \in \{ 2,3 \} $ and let $\Om \subset  \RRR^n$ be a bounded smooth domain. For a real $n<p<\infty$, let $g \in \WW^{1,p}(\Om)$ be a Riemannian metric such that $\triangle_g = g^{ij}\pr_i \pr_j$. Then, for all $1<p'<\infty$ and $u \in W^{2,p'}(\Om)$, we have
\begin{align*}
\Vert u \Vert_{W^{2,p'}(\Om)} \lesssim \Vert \triangle_g u \Vert_{L^{p'}(\Om)} + \Vert u \Vert_{W^{2-1/p',p'}(\pr \Om)}.
\end{align*}
\end{theorem}

From Theorem \ref{cor:p33est24231} we get the following standard higher regularity estimate, see for example Theorem 8.13 in \cite{GilbargTrudinger}.
\begin{corollary}[Standard global higher regularity elliptic estimate] \label{second2241534}
Let $\Om \subset  \RRR^3$ be a bounded smooth domain. Let the Riemannian metric $g \in \HH^2(\Om)$ be such that $\triangle_g = g^{ij} \pr_i \pr_j$. Then for all $u \in H^{3}(\Om)$,
\begin{align*} 
\Vert u \Vert_{H^3(\Om)} 
\lesssim \Vert \triangle_g u \Vert_{H^1(\Om)} + \Vert u \Vert_{H^{5/2}(\pr \Om)}.
 \end{align*}
\end{corollary}

We now turn to the Laplace-Beltrami operator of a general Riemannian metric in three dimensions.
\begin{proposition} \label{control9993434} Let $\Om \subset \RRR^3$ be bounded smooth domain and $g \in \WW^{1,6}(\Om)$ a Riemannian metric. 
Let $u$ be the solution to 
\begin{align*}
\triangle_g u &= f \text{ in }  \Om,\\
u &= 0\text{ on } \pr  \Om.
\end{align*}
for $f \in L^6$. Then $u \in W^{2,6}(\Om)$ and 
\begin{align} \label{eq857674} 
\Vert u \Vert_{W^{2,6}(\Om)}  \lesssim \Vert f \Vert_{L^6(\Om)}. 
\end{align}
\end{proposition}

\begin{proof} First, we note that by the Lax-Milgram theorem and Lemma \ref{sobolevdomainlemma}, there exists a solution $u \in H^1(\Om)$ with
\begin{align*}
\Vert u \Vert_{H^1(\Om)} \lesssim \Vert f \Vert_{H^{-1}(\Om)} \lesssim \Vert f \Vert_{L^6(\Om)},
\end{align*}
and consequently also 
\begin{align} \label{6934934}
\Vert u \Vert_{L^6(\Om)} \lesssim \Vert f \Vert_{L^6(\Om)}.
\end{align}

By Theorem \ref{cor:p33est24231},
\begin{align} \begin{aligned}
\Vert u \Vert_{W^{2,6}(\Om)} \lesssim& \Vert g^{ij} \pr_i \pr_j u \Vert_{L^{6}(\Om)} \\
\lesssim& \Vert f \Vert_{L^{6}(\Om)} + \Vert \triangle_g u-g^{ij} \pr_i \pr_j u \Vert_{L^{6}(\Om)} \\
\lesssim& \Vert f \Vert_{L^{6}(\Om)} + \Vert \pr g \pr u \Vert_{L^{6}(\Om)} \\ 
\lesssim& \Vert f \Vert_{L^{6}(\Om)} + \Vert g \Vert_{\WW^{1,6}(\Om)} \Vert \pr u \Vert_{C^0(\Om)} . 
\end{aligned} \label{qe5828435} \end{align}

In the following, we control the term $\Vert \pr u \Vert_{C^0(\Om)}$ on the right-hand side of \eqref{qe5828435} by a bootstrap argument\footnote{This argument is taken from \cite{Petersen}.}. Using that $g \in \WW^{1,6}(\Om) \subset \CC^0(\Om)$ in $n=3$, we have by partial integration on $\Om$, 
\begin{align} \label{eq93422}
\Vert \pr u \Vert_{L^2(\Om)} \lesssim \Vert \nab u \Vert_{L^2(\Om)} \lesssim \Vert u \Vert^{1/2}_{L^2(\Om)} \Vert \triangle_g u \Vert^{1/2}_{L^2(\Om)} \lesssim \Vert f \Vert_{L^6(\Om)}.
\end{align}
By using Theorem \ref{cor:p33est24231} with $p'=3/2$, this implies further
\begin{align} \begin{aligned}
\Vert u \Vert_{W^{2,3/2}(\Om)} \lesssim& \Vert g^{ij} \pr_i \pr_j u \Vert_{L^{3/2}(\Om)} \\
\lesssim & \Vert  \triangle_g u\Vert_{L^{3/2}(\Om)} + \Vert \pr g \pr u \Vert_{L^{3/2}(\Om)} \\
\lesssim & \Vert  f \Vert_{L^{3/2}(\Om)} + \Vert g \Vert_{\WW^{1,6}(\Om)} \Vert \pr u \Vert_{L^2(\Om)} \\
\lesssim & \Vert  f \Vert_{L^{3/2}(\Om)} + \Vert g \Vert_{\WW^{1,6}(\Om)} \Vert f \Vert_{L^{6}(\Om)} \\
\end{aligned} \label{eq5748293}\end{align}
By Lemma \ref{sobolevdomainlemma} and \eqref{eq5748293}, we have
\begin{align*}
\Vert \pr u \Vert_{L^3(\Om)} \lesssim& \Vert \pr u \Vert_{W^{1,3/2}(\Om)} \\
\lesssim& \Vert f \Vert_{3/2(\Om)} + \Vert g \Vert_{\WW^{1,6}(\Om)} \Vert f \Vert_{L^6(\Om)} \\
\lesssim& \Vert f \Vert_{L^{6}(\Om)}.
\end{align*}
Compared to \eqref{eq93422}, we bootstrapped the regularity of $\pr u$. 

By continuing to bootstrap the regularity of $\pr u$ with Theorem \ref{cor:p33est24231} and Lemma \ref{sobolevdomainlemma} as above (with $p'=2$, $p'=3$ and finally $p'=24/5$), we get
\begin{align*}
\Vert \pr u \Vert_{W^{1,24/5}(\Om) } \lesssim \Vert f \Vert_{L^{6}(\Om)},
\end{align*}  
and consequently, by Lemma \ref{sobolevdomainlemma} and \eqref{6934934}, 
\begin{align*}
\Vert \pr u \Vert_{C^0(\Om)} \lesssim& \Vert \pr u \Vert_{W^{1,24/5}(\Om) }, \\
 \lesssim& \Vert f \Vert_{L^{6}(\Om)}.
\end{align*}
Plugging this bound into \eqref{qe5828435} proves \eqref{eq857674}. This finishes the proof of Proposition \ref{control9993434}.
\end{proof}

\begin{corollary}[Global estimates for the Dirichlet problem in $H^3$, $n=3$] \label{prop:ellest3}
Let $\Om \subset \RRR^3$ be bounded smooth domain and $g \in \HH^2(\Om)$ a Riemannian metric. Let $u$ be the solution to 
\begin{align} \begin{aligned}
\triangle_g u &= f \text{ in }  \Om,\\
u &= h\text{ on } \pr  \Om.
\end{aligned} \label{tildepsidef1255443}\end{align}
for $f \in H^1(\Om)$ and $h \in H^{5/2}(\pr \Om)$. Then we have $u \in H^3(\Om)$ and 
\begin{align*}
\Vert u \Vert_{H^3(\Om)} \lesssim \Vert f \Vert_{H^1(\Om)} + \Vert h \Vert_{H^{5/2}(\pr \Om)}.
\end{align*}
\end{corollary}

\begin{proof} First, we reduce to homogeneous Dirichlet data. By Paragraph 7.41 in \cite{Adams}, there exists $\tilde{h} \in H^3(\Om)$, such that 
\begin{align*}
\tau(\tilde h) = h \text{ on } \pr \Om
\end{align*}
and
\begin{align*}
\Vert \tilde h \Vert_{H^3(\Om)} \lesssim \Vert h \Vert_{H^{5/2}(\pr \Om)}.
\end{align*}

Therefore we can reduce \eqref{tildepsidef1255443} to the study of
\begin{align} \begin{aligned}
\triangle_g u &= \tilde{f} \text{ in }  \Om,\\
u &= 0 \text{ on } \pr  \Om.
\end{aligned} \label{newnewnew}\end{align}
where $\tilde{f} := f + \triangle_g \tilde{h}$ is bounded by
\begin{align} \begin{aligned}
\Vert \tilde{f} \Vert_{H^1(\Om)} \lesssim& \Vert f \Vert_{H^1(\Om)} + \Vert \triangle_g \tilde{h} \Vert_{H^1(\Om)} \\
\lesssim& \Vert f \Vert_{H^1(\Om)} + \Vert g \pr^2 \tilde{h} \Vert_{H^1(\Om)} + \Vert \pr g \pr \tilde{h} \Vert_{H^1(\Om)}\\
\lesssim& \Vert f \Vert_{H^1(\Om)} + \Vert g \Vert_{\HH^2(\Om)}  \Vert \tilde{h} \Vert_{H^3(\Om)} .
\end{aligned} \label{eq923434Gjefef} \end{align}

By the Lax-Milgram theorem there exists a unique solution $u$ to \eqref{newnewnew} satisfying
\begin{align*}
\Vert u \Vert_{H^1(\Om)} \lesssim& \Vert \tilde{ f} \Vert_{L^2(\Om)}.
\end{align*}
In particular, by Lemma \ref{sobolevdomainlemma},
\begin{align*}
\Vert u \Vert_{L^6(\Om)} \lesssim \Vert u \Vert_{H^1(\Om)} \lesssim \Vert \tilde{f} \Vert_{L^2(\Om)}.
\end{align*}
Applying Corollary \ref{second2241534} and Proposition \ref{control9993434} to \eqref{newnewnew}, we get
\begin{align*}
\Vert u \Vert_{H^3(\Om)} \lesssim& \Vert g^{ij} \pr_i \pr_j u \Vert_{H^1(\Om)} + \Vert u \Vert_{H^{5/2}(\pr \Om)} \\
\lesssim& \Vert \triangle_g u \Vert_{H^1(\Om)} + \Vert \pr g \pr u \Vert_{H^1(\Om)} \\
\lesssim &  \Vert \tilde{f} \Vert_{H^1(\Om)} + \Vert g \Vert_{\HH^2(\Om)} \Vert \pr u \Vert_{C^0(\overline{\Om})} \\
\lesssim &  \Vert \tilde{f} \Vert_{H^1(\Om)} + \Vert g \Vert_{\HH^2(\Om)} \Vert u \Vert_{W^{1,6}(\Om)}  \\
\lesssim &  \Vert \tilde{f} \Vert_{H^1(\Om)} + \Vert g \Vert_{\HH^2(\Om)}  \Vert \tilde{f} \Vert_{L^6(\Om)} \\
\lesssim & \Vert \tilde{f} \Vert_{H^1(\Om)}, \\
\lesssim & \Vert f \Vert_{H^1(\Om)} + \Vert h \Vert_{H^{5/2}(\pr \Om)},
\end{align*}
where we used Lemma \ref{sobolevdomainlemma} and \eqref{eq923434Gjefef}. This finishes the proof of Corollary \ref{prop:ellest3}. \end{proof}

\subsection{Global estimates for Neumann data} \label{secNeuEll}

In this section we collect standard global elliptic estimates for given Neumann boundary data. These estimates are applied in Section \ref{sec:harmregularitymflds} to the metric components $g^{33}$ and $g^{3A}$ in boundary harmonic coordinates.

We recall the standard elliptic estimates in a smooth domain, see for example (2.3.3.1) in \cite{Grisvard}.
\begin{theorem} \label{thmoriginalgrisvard}
Let $1<p<\infty$ be a real and $\Om$ be a bounded smooth domain. Let $g \in C^{0,1}(\ol{\Om})$ be a Riemannian metric on $\ol{\Om}$. Let further $B(u):= b^i \pr_i u$ be a boundary operator with $b^i$ Lipschitz on $\pr \Om$ and $g(b,N)\geq c$ for some $c>0$. Then for any $u \in W^{2,p}(\Om)$, we have
\begin{align*}
\Vert u \Vert_{W^{2,p}(\Om)} \lesssim \Vert \triangle_g u \Vert_{L^{p}(\Om)} + \Vert B(u) \Vert_{W^{1-1/p,p}(\pr \Om)}.
\end{align*}
\end{theorem}
In the special case of boundary harmonic coordinates and Neumann boundary data, the proof of Theorem \ref{thmoriginalgrisvard} generalises to the following result.
\begin{theorem}[Global elliptic estimates for the Neumann problem, $n=3$] \label{thm:grisvard1low}
Let $3<p< \infty$ be a real and let $\Om \subset  \RRR^3$ be a bounded smooth domain. Let $g \in \WW^{1,p}(\Om)$ be a Riemannian metric on $\Om$ such that $\triangle_g = g^{ij} \pr_i \pr_j$. Then for every $1<p'<\infty$ and $u \in W^{2,p'}(\Om)$,
\begin{align} \label{est58290111}
\Vert u \Vert_{W^{2,p'}(\Om)} \lesssim  \Vert \triangle_g u  \Vert_{L^{p'}(\Om)} +\Vert N(u) \Vert_{W^{1-1/p',p'}(\pr \Om)}.
\end{align}
\end{theorem}
\begin{proof} The proof of Theorem \ref{thmoriginalgrisvard} in \cite{Grisvard} uses first an analysis of the Neumann problem for a constant coefficient operator and second a continuity argument. We claim that both the same steps go through in the low regularity setting. Indeed, on the one hand, by Lemma \ref{sobolevdomainlemma} and $p>3$,
\begin{align*}
g^{ij} \in W^{1,p} \subset C^{0,\a}
\end{align*}
for some $\a>0$. Therefore, at each point, $g^{ij} \pr_i \pr_j$ is a well-defined pointwise elliptic operator. 

On the other hand, inspecting the continuity argument in the proof of Theorem \ref{thmoriginalgrisvard} in \cite{Grisvard}, it follows that in the case $\triangle_g = g^{ij} \pr_i \pr_j$ and $B=N$, the regularity $g^{ij} \in C^{0,\a}$ suffices to derive \eqref{est58290111}, see in particular the estimates (2.3.3.5) and (2.3.3.6) in \cite{Grisvard} where the regularity of $g^{ij}$ is used. This finishes the proof of Theorem \ref{thm:grisvard1low}. \end{proof}


\section{Interior elliptic estimates} In this section we derive interior elliptic estimates in fractional regularity and control the Dirichlet problem for the Laplace-Beltrami operator.

We recall the following interior estimates, see Proposition 3.4 in \cite{Maxwell}.
\begin{theorem}[Interior estimates in $H^{3/2}$] \label{fracellestlit}
Let $\Om' \subset \subset \Om \subset \RRR^2$ be two bounded smooth domains, and let $g \in \HH^{3/2}(\Om)$ be a Riemannian metric. Then for all $u \in H^{3/2}(\Om)$,
\begin{align*}
\Vert u \Vert_{H^{3/2}(\Om')} \lesssim \Vert \triangle_g u \Vert_{H^{-1/2}(\Om)} + \Vert u \Vert_{H^{1/2}(\Om)}.
\end{align*}
\end{theorem}
Moreover, we have the following interior estimates. They are derived analogously to Proposition \ref{control9993434}, see also Theorem A.3 in \cite{Petersen}.
\begin{proposition}[Interior estimates in $W^{2,4}$] \label{thmA3Petersen}
Let $\Om' \subset \subset \Om \subset \RRR^2$ be two bounded smooth domains. Let $g \in \WW^{1,4}(\Om)$ be a Riemannian metric. Then for all $u \in W^{2,4}(\Om)$,
\begin{align*}
\Vert u \Vert_{W^{2,4}(\Om')} \lesssim \Vert \triangle_g u \Vert_{L^{4}(\Om)} + \Vert u \Vert_{L^{4}(\Om)} .
\end{align*}
\end{proposition}
By combining Theorem \ref{fracellestlit} and Proposition \ref{thmA3Petersen}, we have the following stronger interior elliptic estimates for the Laplace-Beltrami operator.
\begin{corollary}[Interior estimates in $H^{5/2}$, $n=2$] \label{propdim2higher}
Let $\Om' \subset \subset \Om \subset \RRR^2$ be two bounded smooth domains. Let $g \in \HH^{3/2}(\Om)$ be a Riemannian metric. Then for all $u \in H^{5/2}(\Om)$, we have
\begin{align} \label{eq393493422}
\Vert u \Vert_{H^{5/2}(\Om')} \lesssim \Vert \triangle_g u \Vert_{H^{1/2}(\Om)} + \Vert u \Vert_{H^{3/2}(\Om)}.
\end{align}
\end{corollary}

\begin{proof} Let $g \in \HH^{3/2}(\Om)$ and let 
$$\Om' \subset \subset \Om_2 \subset \subset \Om_1 \subset \subset \Om$$
be bounded open subsets with smooth boundary. By Theorem \ref{fracellestlit}, 
\begin{align} \label{eq45020064}
\Vert u \Vert_{H^{3/2}(\Om_1)} \lesssim \Vert \triangle u \Vert_{{H^{-1/2}(\Om)}} + \Vert u \Vert_{H^{1/2}(\Om)}.
\end{align}

By Proposition \ref{thmA3Petersen} and \eqref{eq45020064},
\begin{align} \begin{aligned}
\Vert u \Vert_{W^{2,4}(\Om_2)} &\lesssim \Vert \triangle_g u \Vert_{L^4(\Om_1)} + \Vert u \Vert_{L^4(\Om_1)} \\
&\lesssim \Vert  \triangle_g u \Vert_{H^{1/2}(\Om)} + \Vert u \Vert_{H^{1/2}(\Om)}.
\end{aligned} \label{lemest2} \end{align}

Applying now Theorem \ref{fracellestlit} to $\pr u$ and using \eqref{eq45020064} and \eqref{lemest2}, we get
\begin{align} \begin{aligned}
\Vert \pr u \Vert_{H^{3/2}(\Om')} \lesssim& \Vert \triangle_g \pr u \Vert_{H^{-1/2}(\Om_2)} + \Vert \pr u \Vert_{H^{1/2}(\Om_2)} \\
\lesssim& \Vert \triangle_g u \Vert_{H^{1/2}(\Om_2)} +\Vert [\triangle_g, \pr ] u \Vert_{H^{-1/2}(\Om_2)} +  \Vert u \Vert_{H^{3/2}(\Om_2)}.
\end{aligned} \label{calcellest3} \end{align}

The commutator term equals schematically 
\begin{align*}
[\triangle_g, \pr] u = g^2 \pr^2 g \pr u + g^2 \pr g \pr^2 u.
\end{align*}
By Lemma \ref{sobolevdomainlemma}, the first term of the commutator is bounded by \eqref{lemest2} as 
\begin{align*}
\Vert g^2 \pr^2 g \pr u \Vert_{H^{-1/2}(\Om_2)} \lesssim& \Vert \pr^2 g \pr u \Vert_{H^{-1/2}(\Om_2)} \\
\lesssim& \Vert \pr^2 g \Vert_{\HH^{-1/2}(\Om_2)} \Vert \pr u \Vert_{W^{1,4}(\Om_2)} \\
\lesssim& \Vert g \Vert_{\HH^{3/2}(\Om_2)} \Vert u \Vert_{W^{2,4}(\Om_2)} \\
\lesssim& \Vert g \Vert_{\HH^{3/2}(\Om)} \Big( \Vert \triangle_g u \Vert_{H^{1/2}(\Om)} + \Vert u \Vert_{H^{1/2}(\Om)} \Big).
\end{align*}
Similarly, the second term of the commutator is bounded by
\begin{align*}
\Vert g^2 \pr g \pr^2 u \Vert_{H^{-1/2}(\Om_2)} \lesssim& \Vert \pr g \pr^2 u \Vert_{H^{-1/2}(\Om_2)} \\
\lesssim& \Vert \pr g \Vert_{\HH^{1/2}(\Om_2)} \Vert \pr^2 u \Vert_{W^{1,4}(\Om_2)} \\
\lesssim& \Vert g \Vert_{\HH^{3/2}(\Om)} \Big(\Vert \triangle_g u \Vert_{H^{1/2}(\Om)} + \Vert u \Vert_{H^{1/2}(\Om)} \Big).
\end{align*}
Plugging this into \eqref{calcellest3} and summing over all coordinate derivatives proves \eqref{eq393493422}. This finishes the proof of Proposition \ref{propdim2higher}. \end{proof}


The proof of the next corollary of Proposition \ref{propdim2higher} follows by the Lax-Milgram theorem and is left to the reader.
\begin{corollary}[Interior estimates for the Dirichlet problem in $H^{5/2}$, $n=2$] \label{fullldg6858}
Let $\Om \subset \RRR^2$ be a bounded smooth domain and let $g \in H^{3/2}(\Om)$ be a Riemannian metric on $\Om$. Then for every $f \in H^{1/2}(\Om)$, there exists a unique solution $u$ to
\begin{align*} 
\triangle_g u &= f \text{ in }  \Om,\\
u &= 0\text{ on } \pr  \Om.
\end{align*}
Moreover, for every smooth domain $\Om' \subset\subset \Om$, we have $u \in H^{5/2}(\Om')$ and
\begin{align*}
\Vert u \Vert_{H^{5/2}(\Om')} \lesssim \Vert f \Vert_{H^{1/2}(\Om)}.
\end{align*} 
\end{corollary}


\end{document}